\documentclass[12pt]{amsart}
\usepackage{amsmath}
\usepackage{mathrsfs}
\usepackage{psfrag}
\usepackage[cmtip,line,color,poly,all]{xy}
\usepackage{graphicx,amssymb,booktabs, verbatim, mathtools}
\usepackage{stackrel}
\usepackage{rotating}
\usepackage[noadjust]{cite}
\usepackage{xspace}
\usepackage{blindtext}

\usepackage[pdfusetitle, pdfpagelabels, plainpages=false,
  bookmarks, bookmarksnumbered
]{hyperref} 
\usepackage{bookmark}
\usepackage{subfigure}
\newif\iffloattoend

\iffloattoend
\usepackage[figuresonly, nomarkers, nolists, heads]{endfloat}

\fi
\newif\ifvc
\vcfalse
\ifvc
\input{vc}
\fi

\newcommand{\RR}{\mathbb R}
\newcommand{\ZZ}{\mathbb Z}
\newcommand{\QQ}{\mathbb Q}

\newcommand{\n}{\mathbf{n}}

\newcommand{\ip}[1]{\left \langle #1 \right \rangle}

\newcommand{\sets}[1]{[\![#1]\!]}

\newcommand{\sA}{\mathscr{A}}
\newcommand{\TT}{\mathscr{T}}
\newcommand{\sX}{\mathscr X}%

\newcommand{\sP}{\mathscr{P}}
\newcommand{\sC}{\mathscr{C}}

\DeclareMathOperator{\Tr}{Tr}

\swapnumbers
\theoremstyle{plain}
\newtheorem{thm}[subsection]{Theorem}

\newtheorem{prop}[subsection]{Proposition}
\newtheorem{conj}[subsection]{Conjecture}

\theoremstyle{definition}
\newtheorem{defn}[subsection]{Definition}
\newtheorem{remark}[subsection]{Remark}

\catcode`\@=11
\def\@secnumfont{\bfseries}
\catcode`\@12

\makeatletter

\begin{document}

\title[Hypergraph Catalan numbers]{Generalized Catalan
numbers from hypergraphs}

\author{Paul E. Gunnells}
\address{Department of Mathematics and Statistics\\University of
Massachusetts\\Amherst, MA 01003-9305}
\email{gunnells@math.umass.edu}

\renewcommand{\setminus}{\smallsetminus}

\date{2 April 2019} 

\thanks{The author was partially supported by NSF grant DMS 1501832.
We thank Dan Yasaki and two referees for many helpful comments}

\keywords{Matrix models, hypergraphs, Catalan numbers}

\subjclass[2010]{Primary 05A10, 11B65}

\begin{abstract}
The Catalan numbers $(C_{n})_{n\geq 0} =  1,1,2,5,14,42,\dots$ form one of
the most venerable sequences in combinatorics.  They have many
combinatorial interpretations, from counting bracketings of products
in non-associative algebra to counting plane trees and
noncrossing set partitions.  They also arise in the GUE matrix model
as the leading coefficient of certain polynomials, a connection
closely related to the plane trees and noncrossing set partitions
interpretations.  In this paper we define a generalization of the
Catalan numbers.  In fact we define an infinite collection of
generalizations $C_{n}^{(m)}$, $m\geq 1$, with $m=1$ giving 
the usual Catalans. The sequence $C_{n}^{(m)}$ comes from
studying certain matrix models attached to hypergraphs.  We also give
some combinatorial interpretations of these numbers, and conjecture
some asymptotics.
\end{abstract}

\maketitle

\ifvc
\let\thefootnote\relax
\footnotetext{Base revision~\GITAbrHash, \GITAuthorDate,
\GITAuthorName.}
\fi

\section{Introduction}\label{s:intro}

\subsection{}\label{ss:introinterps}
The \emph{Catalan numbers} $(C_{n})_{n \geq 0}$
\[
1, 1, 2, 5, 14, 42, 132, 429, 1430, \dotsc ,
\]
form one of the most venerable sequences in combinatorics.  
They have many combinatorial interpretations, far more than can be
reproduced here.  We only mention a few:
\begin{enumerate}
\item A \emph{plane tree} is a rooted tree with an ordering specified
for the children of each vertex.  Then $C_{n}$ counts the number of
plane trees with $n+1$ vertices. \label{item:planetrees}
\item A \emph{Dyck path} of length $n$ is a directed path from $(0,0)$
to $(n,0)$  in $\RR^{2}$ that only uses steps of type $(1,1)$ and
$(1,-1)$ and never crosses below the $x$-axis.  Then $C_{n}$
counts the number of Dyck paths of length $2n$. \label{item:dyck}
\item A \emph{ballot sequence} of length $2n$ is a sequence
$(a_{1},\dotsc ,a_{2n})$ with $a_{i}\in \{\pm 1 \}$ with total sum $0$
and with all partial sums nonnegative. Then $C_{n}$ counts the number
of ballot sequences of length $2n$. \label{item:ballot}
\item A \emph{binary plane tree} is the empty graph or a plane tree in
which every node $N$ has at most two children, which are called the
left and right children of $N$.  Furthermore, if a node has only one
child, then it must be either a left or right child.
Then $C_{n}$ counts the number of binary plane trees
with $n$ vertices. \label{item:binaryplanetrees}
\item A regular $(n+2)$-gon can be subdivided into triangles without
adding new vertices by drawing $n-1$ new diagonals.  Then $C_{n}$
counts these subdivisions. \label{item:subdivisions}
\item Let $\Pi$ be a polygon with $2n$ sides.  A \emph{pairing} of
$\Pi$ is a partition of the edges of $\Pi$ into blocks of size $2$ and
a choice of relative orientations for the edges in each block.  Any
pairing $\pi$ of $\Pi$ determines a compact topological surface
$\Sigma_{\pi}$ of some genus: one identifies the edges together
according to the pairing.  One can show that the surface $\Sigma_{\pi
}$ is orientable if and only if the edges in each pair are have
opposite orientations as one walks around the boundary of $\Pi$
(cf.~\cite[Ch.~1]{massey}). Then $C_{n}$ counts the number of pairings
of the sides of $\Pi$ such that $\Sigma_{\pi}$ is orientable and has
genus $0$, i.e.~is homeomorphic to the $2$-sphere.
\label{item:gluings}
\end{enumerate}


The definitive reference for combinatorial interpretations of Catalan
numbers is Richard Stanley's recent monograph \cite{catalan}.  It
contains no fewer than $214$ different interpretations of the
$C_{n}$.(\footnote{An earlier version of this list is contained in
\cite{ec2}, with additions available on Stanley's website
\cite{addendum}.}) Indeed, the first five interpretations given above
appear in \cite[Ch.~2]{catalan} as items (6), 
(25), 
(77), 
(4), 
and 
(1) 
respectively.  The last
interpretation (counting genus 0 polygon gluings) is unfortunately not
in \cite{catalan}.  However, it is easily seen to be equivalent to
\cite[Ch.~2, (59)]{catalan}, which counts the number of ways to draw $n$ nonintersecting
chords joining $2n$ points on the circumference of a circle.  
Another resource is OEIS \cite{oeis}, where the
Catalans are sequence \texttt{A000108}.

\subsection{}\label{ss:catexamples} The goal of this paper is to give a family of
generalizations of the $C_{n}$.  For each integer $m\geq 1$, we define
a sequence of integers $(C_{n}^{(m)})_{n\geq 0}$; for $m=1$ we have
$C_{n}^{(1)} = C_{n}$.  Here are some
further examples:
\begin{align*}
m = 2:&\quad 1, 1, 6, 57, 678, 9270, 139968, 2285073, 39871926, 739129374, 14521778820,\dotsc  \\
m = 3:&\quad 1, 1, 20, 860, 57200, 5344800, 682612800, 118180104000, 27396820448000, \dotsc  \\
m = 4:&\quad 1, 1, 70, 15225, 7043750, 6327749750, 10411817136000, 29034031694460625,\dotsc  \\
m = 5:&\quad 1, 1, 252, 299880, 1112865264, 11126161436292, 255654847841227632, \dotsc  \\
m = 6:&\quad 1, 1, 924, 6358044, 203356067376, 23345633108619360, \dotsc  \\
m = 7:&\quad 1, 1, 3432, 141858288, 40309820014464, 53321581727982247680, \\
      &\quad \quad \quad 238681094467043912358445056,\dotsc 
\end{align*}

The $C_{n}^{(m)}$ are defined in terms of counting walks on trees,
weighted by the orders of their automorphism groups.  For $m=1$ the
resulting expression is not usually given as a standard combinatorial
interpretation of the Catalan numbers, but it is known to compute
them; we will prove it in the course of proving Theorem
\ref{thm:interpretations}.  In fact, from our definition it is not
clear that the $C_{n}^{(m)}$ are actually \emph{integers}, even for
$m=1$, although we will see this by giving several combinatorial
interpretations of them.

Here is the plan of the paper.  In \S\ref{s:defofgencat} we give the
definition of the $C_{n}^{(m)}$, and in \S\ref{s:computing} we explain
how to compute them for moderate values of $n$ and any $m$.  In \S
\ref{s:ci} we give six different combinatorial interpretations of the
$C_{n}^{(m)}$ based on six standard interpretations of the Catalan
numbers.  In \S \ref{s:gf} we explain how to compute the generating
function of the $C_{n}^{(m)}$, and conjecture some asymptotics
of $C_{n}^{(m)}$ for fixed $m$ as $n\rightarrow \infty$.  Finally, in \S
\ref{s:mm} we discuss how these numbers arise in the study of
certain matrix models associated to hypergraphs.

\section{Hypergraph Catalan Numbers}\label{s:defofgencat}

\subsection{} We begin by giving the description of the Catalan
numbers that we wish to generalize.  Let $\TT_{n}$ be the set of
unlabeled trees on $n$ vertices.  The sequence $|\TT_{n}|$
appears on {OEIS} as sequence \texttt{A000055}, and begins
\[
1, 1, 1, 1, 2, 3, 6, 11, 23, 47, 106, \dotsc ,
\]
where $|\TT_{0}| := 1$ by convention.  

Let $T\in \TT_{n}$, and for each vertex $v\in T$, let $a_{T}(v)$ be
the number of walks that begin and end at $v$ and traverse each edge
of $T$ exactly twice.  Note that any such walk visits each other
vertex at least once, and may do so multiple times.  Let $\Gamma (T)$
be the automorphism group of $T$, and let $|\Gamma (T)|$ be its order.

Figure \ref{fig:trees} shows an example of the numbers $a_{T } (v)$
for the 3 trees in $\TT_{5}$.  The outer numbers on the
leaves of the upper left tree are $1$ because the only possible walk
is to go from one end of the tree to the other, then back to the
beginning.  The inner numbers on the same tree are $2$, because
one must first choose a direction in which to head, then must go all
the way to that end, then back through the starting point to the other
end, then back to the initial vertex.

\begin{prop}\label{prop:Cnone}
The Catalan number $C_{n}$ is given by 
\begin{equation}\label{eq:catdef}
C_{n} = \sum_{T\in \TT_{n+1}} \sum_{v\in T} \frac{a_{T} (v)}{|\Gamma (T)|}.
\end{equation}
\end{prop}
For example, using Figure \ref{fig:trees} we have 
\[
C_{4} = \frac{8}{2}
 + \frac{48}{24} + \frac{16}{2} = 4+2+8 = 14.
\]

We defer the verification of Proposition \ref{prop:Cnone} until
\S\ref{s:ci}, when we discuss combinatorial interpretations.  The
proposition will be proved in the course of Theorem
\ref{thm:interpretations}.

\begin{figure}[htb]
\psfrag{1}{$\scriptstyle 1$}
\psfrag{2}{$\scriptstyle 2$}
\psfrag{4}{$\scriptstyle 4$}
\psfrag{6}{$\scriptstyle 6$}
\psfrag{24}{$\scriptstyle 24$}
\begin{center}
\includegraphics[scale=0.25]{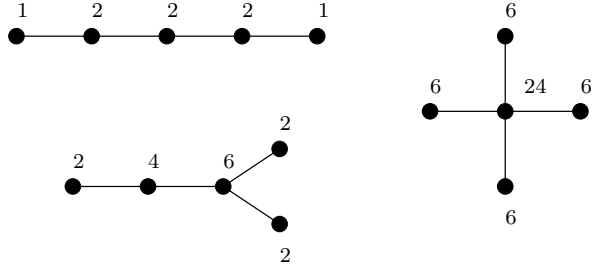}
\end{center}
\caption{The set $\TT_{5}$ with vertices labeled by $a_{T}
(v)$\label{fig:trees}.  Going clockwise from the upper left, the
orders of the automorphism groups are $2, 24, 2$.}
\end{figure}

\subsection{} Now let $m\geq 1$ be a positive integer.  Then
$C^{(m)}_{n}$ is defined essentially as in \eqref{eq:catdef}, but we
modify the definition of the numbers $a_{T} (v)$:

\begin{defn}\label{def:amTtour}
Let $m\geq 1$, let $T\in
\TT_{n+1}$, and let $v\in T$.  Then an \emph{$a^{(m)}_{T}$-tour}
beginning at $v$ is a walk that begins and ends at $v$ and traverses
each edge of $T$ exactly $2m$ times.  We denote by $a^{(m)}_{T} (v)$
the number of $a^{(m)}_{T}$-tours beginning at $v$.  
\end{defn}

We note that $a^{(1)}_{T} (v) = a_{T} (v)$.  As before, in an
$a^{(m)}_{T}$-tour each vertex of $T$ will be visited at least once.
Furthermore, since $T$ is a tree, each edge is visited $m$ times while
going away from $v$ and $m$ times while coming back to $v$.

\begin{defn}\label{def:gencat}
The \emph{hypergraph Catalan numbers} $C_{n}^{(m)}$ are defined by
\begin{equation}\label{eq:catdef2}
C^{(m)}_{n} = \sum_{T\in \TT_{n+1}} \sum_{v\in T} \frac{a_{T}^{(m)} (v)}{|\Gamma (T)|}.
\end{equation}
\end{defn}

For example the numbers $a_{T}^{(2)} (v)$ are shown in Figure
\ref{fig:trees2} for the three trees in $\TT_{5}$.  The numbers are
larger now, since walks have many more options.
Using the numbers in Figure \ref{fig:trees2} we find
\[
C_{4}^{(2)} = \frac{216}{2} + \frac{5040}{24} + \frac{720}{2} = 108+210+360 = 678.
\]

\begin{figure}[htb]
\psfrag{27}{$\scriptstyle 27$}
\psfrag{54}{$\scriptstyle 54$}
\psfrag{630}{$\scriptstyle 630$}
\psfrag{180}{$\scriptstyle 180$}
\psfrag{270}{$\scriptstyle 270$}
\psfrag{90}{$\scriptstyle 90$}
\psfrag{2520}{$\scriptstyle 2520$}
\begin{center}
\includegraphics[scale=0.25]{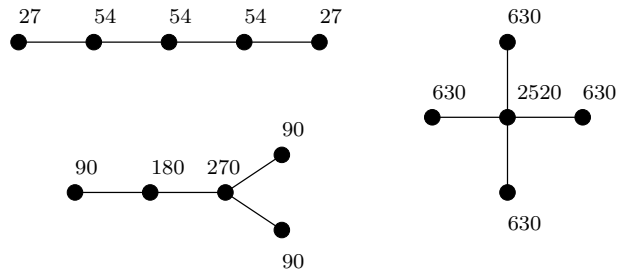}
\end{center}
\caption{The set $\TT_{5}$ with vertices labeled by $a_{T}^{(2)}
(v)$.\label{fig:trees2}}
\end{figure}

\section{Computing the $C_{n}^{(m)}$}\label{s:computing}

\subsection{}
In this section we show that once one has sufficient knowledge of the
trees in $\TT_{n+1}$ to compute $C_{n}$, one can easily compute
$C^{(m)}_{n}$ for any $m$.  In other words, if one wants to extend the
data in \S\ref{ss:catexamples}, we show that it is
easy to fix $n$ and let $m$ grow.

\begin{thm}\label{thm:computingcat}
Let $T\in \TT_{n+1}$, and let $v \in T$ have degree $d (v)$.
Then we have
\begin{equation}\label{eq:contrib}
\sum_{v\in T} {a_{T}^{(m)} (v)}=
\frac{2nm^{n+1}}{(m!)^{2n}} \prod_{v\in T}(md (v) -1)!.
\end{equation}
\end{thm}

For example, for the tree on the right of Figure \ref{fig:trees2}, we
find 
\[
\frac{2\cdot 4 \cdot 2^{5}}{(2!)^{8}} (2\cdot 1-1)!^{4} (2\cdot 4-1)!
=  5040,
\]
which agrees with the data in the figure.

\begin{proof}
We use the results in \cite[Ch.~10]{stanley.alg.comb}, which treat
Eulerian tours in balanced digraphs, and so we begin by recalling some
notation.  Let $G$ be a connected digraph and let $\tilde{G}$ be the
associated undirected graph. Suppose $G$ is \emph{balanced}; this
means that the outdegree $o (v)$ of each vertex $v$ is equal to its
indegree $i (v)$.  Given a vertex $v$ of $G$, an oriented spanning
tree with root at $v$ is a subgraph $T\subset G$ such that $\tilde{T}$
is a spanning tree of $\tilde{G}$ in the usual sense, and such that
all the edges of $T$ are oriented towards $v$.  Then since $G$ is
connected and balanced it is Eulerian \cite[Theorem
10.1]{stanley.alg.comb}, and according to
\cite[Theorem 10.2]{stanley.alg.comb} given an edge $e$ the number
$\varepsilon (G,e)$ of Eulerian tours of $G$ beginning with the
(directed) edge $e$ is
\begin{equation}\label{eq:epsGe}
\varepsilon (G,e) = \tau (G, e) \prod_{v\in G} (o (v)-1)!,
\end{equation}
where $\tau (G,e)$ is the number of oriented spanning trees of $G$
with root at the initial vertex of $e$.  Furthermore, it is known that
$\tau (G,e)$ is independent of $e$ \cite[Corollary
10.3]{stanley.alg.comb}.

Now let $T$ be a tree with $n+1$ vertices and fix $m$.  The left of
\eqref{eq:contrib} is the total number of $a^{(m)}_{T}$-tours on $T$
(Definition \ref{def:amTtour}).  We will count $a^{(m)}_{T}$-tours by
first counting Eulerian tours on the canonical balanced digraph
$T_{m}$ built from $T$ by replacing each edge with $2m$ edges, $m$
oriented in one direction and $m$ oriented in the other.  Let $v\in
T$.  Then it is clear that an $a_{T}^{(m)}$-tour contributing to $a^{(m)}_{T} (v)$
determines (non-uniquely) an Eulerian tour on $T_{m}$ starting and
ending at $v$.  Hence we can use \eqref{eq:epsGe} to compute
$a_{T}^{(m)} (v)$.  In particular for an edge $e\in T_{m}$ we have
\begin{equation}\label{eq:epsTme}
\varepsilon(T_{m}, e) = \tau (T_{m}, e) \prod_{v\in T} (md (v)-1)!,
\end{equation}
where $d (v)$ is the degree of $v$ in $T$.

Now we go from \eqref{eq:epsTme} to \eqref{eq:contrib}.  First, the
number $\tau (T_{m}, e)$ of oriented spanning trees in $T_{m}$ is
$m^{n}$ (after fixing a root we pick one of $m$ possible properly
oriented edges for each edge in $T$).  Next, to get the total number
of Eulerian tours of $T_{m}$, we multiply $\varepsilon(T_{m}, e)$ by
the number of edges of $T_{m}$, which is $2mn$.  The result is
\begin{equation}\label{eq:tmtours}
\Bigl|\Bigl\{\text{Eulerian tours of
$T_{m}$}\Bigr\}\Bigr| = 2mn\cdot m^{n} \cdot \prod_{v\in T} (md (v)-1)!.
\end{equation}
Now let $\pi$ be the map
\[
\pi \colon \Bigl\{\text{Eulerian tours of
$T_{m}$}\Bigr\}\longrightarrow \Bigl\{ \text{$a_{T}^{(m)}$-tours of $T$} \Bigr\}
\]
that replaces each edge of an Eulerian tour of $T_{m}$ with the
corresponding edge in $T$.  It is clear that $\pi$ is surjective.
Furthermore, each $a^{(m)}_{T}$-tour $w$ of $T$ has precisely
$(m!)^{2n}$ preimages under this map, since $w$ traverses each edge of
$T$ in each direction precisely $m$ times, and we get to choose in
which order these $m$ traversals correspond to the $m$ corresponding
edges of $T_{m}$.  Thus
\begin{equation}\label{eq:tmandam}
\Bigl|\Bigl\{\text{Eulerian tours of
$T_{m}$}\Bigr\}\Bigr| = (m!)^{2n}\cdot \Bigl| \Bigl\{ \text{$a_{T}^{(m)}$-tours of $T$} \Bigr\}\Bigr|.
\end{equation}
Comparing \eqref{eq:tmtours} and \eqref{eq:tmandam}
yields \eqref{eq:contrib}, and completes the proof.
\end{proof}

\subsection{} Thus to compute $C_{n}^{(m)}$ for any $m$ one only needs
the trees in $\TT_{n+1}$ together with their vertex degrees and orders
of their automorphism groups.  This can be done, at least for
reasonable values of $n$, using the software \texttt{nauty}
\cite{nauty}.  For example there are $751065460 \approx 2^{29}$ trees
in $\TT_{27}$; \texttt{nauty} is able to compute them on a laptop in
less than 14 seconds.  On the other hand, computing the orders of all
the automorphism groups takes longer.  For instance there
are only $823065 \approx 2^{20}$ trees in $\TT_{20}$, and computing
all their automorphism groups takes just over 4 hours.

\section{Combinatorial interpretations}\label{s:ci}

\subsection{} As it turns out, the numbers $C_{n}^{(m)}$ have a
variety of combinatorial interpretations, in fact in terms of objects
used to count the usual Catalan number $C_{nm}$.  As we shall see,
only certain of these will contribute to $C_{n}^{(m)}$, and in general
a given object may contribute in several different ways.  It will also
be evident that any standard Catalan interpretation can be turned into
one for the $C_{n}^{(m)}$.  We begin by introducing some notation.

\subsection{} Let $X$ be an combinatorial object; we do not give a
precise definition of $X$, but the reader should imagine that $X$ is
something used in a standard Catalan interpretation, such as those in
\S \ref{ss:introinterps}.  We will give examples in Theorem
\ref{thm:interpretations}.  Typically $X$ will be an aggregate of
smaller elements, and we say that a \emph{level structure} for $X$ is
a surjective map $\ell$ from these elements to a finite set
$\sets{N} := \{1,2,\dotsc ,N \}$ for some $N\in \ZZ_{>0}$.  We will say $x\in X$
is on a \emph{higher level} than $x'\in X$ if $\ell (x) > \ell (x')$,
with similar conventions for \emph{same} and \emph{lower} level.  The
$i$th level $X_{i}\subset X$ with respect to $\ell$ will be the
preimage $\ell^{-1} (i) \subset X$.  In some interpretations, we will
also have a zeroth level $X_{0}$; these will usually be combinatorial
objects that are naturally rooted.  If $X$ has a zeroth level, we will
require $|X_{0}| = 1$. Note that in the pictures that follow, the
function $\ell$ will not typically correspond to the height of
elements of $X$ in their positions in the figures.

\subsection{} The object $X$ will be a poset.  If $x,x'\in X$ and $x$
covers $x'$, we will say that $x$ is a \emph{parent} of the
\emph{child} $x'$.  The level structures we consider will always be
compatible with the poset structure, in that any child $x'$ of a given
parent $x$ will satisfy $\ell (x') = \ell (x)+1$.  In other words,
children will always lie one level above their parents.  We remark the
poset structure determines the level of each element as long as a
level-$0$ element is present.

\subsection{} Finally, let $m$ be a positive integer.  We will
consider \emph{$m$-labeling} the positive levels of $X$, which means
the following.  First fix an infinite set $L$ of labels.  Let
\[
X = X_{0} \sqcup \bigsqcup_{i\geq 1} X_{i}
\]
be the disjoint decomposition of $X$ into levels, where $X_{0}$ may be
empty.  For each \emph{positive} level $X_{i}$, $i>0$, we choose a
disjoint decomposition of $X_{i}$ into subsets of order $m$; in
particular, this implies $|X_i| \equiv 0 \bmod m$ for $i > 0$ in all
cases we consider.  Then an $m$-labeling is a map assigning an
element of $L$ to each of these subsets.  We will say that an
$m$-labeling is \emph{admissible} if the following are true:
\begin{itemize}
\item Distinct subsets receive distinct labels.
\item The labeling is compatible with the poset structure, in the
following sense: if two elements $x,x'$ share a given label, then the labels
of their parents agree.
\end{itemize}
In other words, an admissible $m$-labeling is a partition of the set
$X\smallsetminus X_0$ of all non-level-0 elements of $X$ into disjoint
size $m$ blocks such that if two elements $u, v \in X \smallsetminus
X_0$ have children in the same block, then $u$ must belong to the same
block as $v$.  We also consider two labelings to be equivalent if one
is obtained from the other by permuting labels.  Note that all
$m$-labelings are admissible if $m=1$.

\subsection{} We give an example to clarify this terminology. Let
$X$ be a plane tree.  The elements of $X$ are its
vertices.  Let $v$ be the root.  We can define a level structure
$\ell \colon X \rightarrow \ZZ_{\geq 0}$ by setting $\ell (x)$ to be
the distance in $X$ to $v$.  Note that $X_{0} = \{v \}$ has
size $1$, but of course the positive levels $X_{i}$ can be bigger.
Given two vertices $x,x'$, one is a parent or a child of the other if
it is in the usual sense of trees: $x$ is a parent (respectively,
child) of $x'$ if $x$ and $x'$ are joined by an edge and $x$ lies
closer to (respectively, further from) the root than $x'$.  Figure
\ref{fig:trees-unlabeled} shows two plane trees $X, Y$.
Each tree has four levels.  Parents appear above their
children, and levels increase as we move down the figure.

Next we consider labelings.  Figure \ref{fig:2-labelings} shows the
two rooted trees $X, Y$ equipped with $2$-labelings, with labeling
set $L=\{a,b,c,d,e \}$.  We have arbitrarily ordered the labels as
indicated; one would obtain an equivalent labeling after permuting the
labels.  The left tree $X$ is admissibly $2$-labeled: if two vertices have
the same label, so do their parents.  The right tree $Y$, however, is
not admissibly $2$-labeled: the two vertices at the bottom have the same
label $e$, but their parents have two different labels $c, d$.

\subsection{} Now let $\sX_{nm}$ be a set of objects constituting a
combinatorial interpretation of $C_{nm}$ (we will say which we
consider in a moment).  Then we will define a poset structure and
levels on each $X \in \sX_{nm}$, and show
\begin{equation}\label{eq:wlabels}
C_{n}^{(m)} = \sum_{X\in \sX_{nm}} N_{m}(X),
\end{equation}
where $N_{m} (X)$ is the number of admissible $m$-labelings of $X$.
As mentioned before, an object $X \in \sX_{nm}$ cannot have $N_{m}
(X)\not =0$ unless its level structure satisfies an obvious congruence
condition: the size of a nonzero level must be divisible by $m$.  This
condition is not sufficient, however, as one can see in Figure
\ref{fig:2-labelings}.  Both trees have positive levels of even size,
but there is no way to give the right tree an admissible $2$-labeling:
the parents of the vertices labeled $e$ are forced to have different
labels.  On the other hand, if $m=1$ then any $m$-labeling is
admissible.  Hence when $m=1$ each $X\in \sX_{nm}$ contributes to
$C_{nm}$ with $N_{m} (X)=1$, and one recovers a usual combinatorial
interpretation of the Catalan numbers.

\begin{figure}[htb]
\psfrag{X0}{$X_{0}$}
\psfrag{X1}{$X_{1}$}
\psfrag{X2}{$X_{2}$}
\psfrag{X3}{$X_{3}$}
\psfrag{Y0}{$Y_{0}$}
\psfrag{Y1}{$Y_{1}$}
\psfrag{Y2}{$Y_{2}$}
\psfrag{Y3}{$Y_{3}$}
\begin{center}
\includegraphics[scale=0.25]{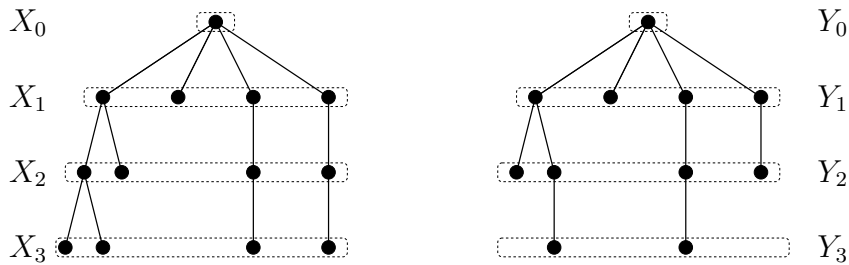}
\end{center}
\caption{Two rooted trees $X, Y$ with their levels.\label{fig:trees-unlabeled}}
\end{figure}

\begin{figure}[htb]
\psfrag{a}{$a$}
\psfrag{b}{$b$}
\psfrag{c}{$c$}
\psfrag{d}{$d$}
\psfrag{e}{$e$}
\psfrag{f}{$f$}
\begin{center}
\includegraphics[scale=0.25]{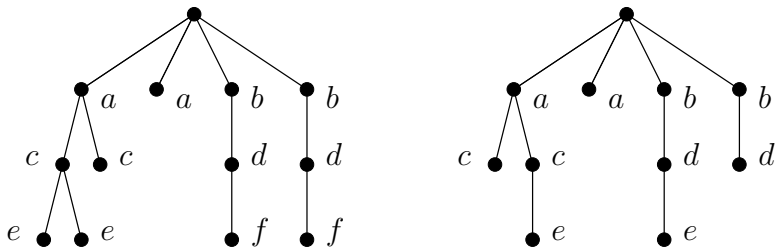}
\end{center}
\caption{$2$-labelings of $X$ and $Y$.  The labeling of $Y$ is not admissible.\label{fig:2-labelings}}
\end{figure}

\subsection{}\label{ss:interp}
We are now ready to give our combinatorial interpretations.  For each
we explain their level and hierarchical structures.  Examples of all
the objects are shown in Figures \ref{fig:c22a}--\ref{fig:c22b}.

\subsection*{(i) Plane trees.}  The set $\sX_{nm}$ is the
set of plane trees on $nm+1$ vertices.  The levels are the distance to
the root, and the vertices are parents/children of each other if they
are in the usual sense.

\subsection*{(ii) Dyck paths.} The set $\sX_{nm}$ is the set of Dyck
paths from $(0,0)$ to $(2mn, 0)$.  The elements of a path $\pi $ are its
\emph{slabs}, defined as follows.  Let $R_{\pi }$ be the interior of the
region bounded by $\pi $ and the $x$-axis.  Then a slab is a connected
component of the intersection of $R_{\pi }$ with an open strip $Y_{k} =
\{(x,y)\in \RR^{2} \mid k-1<y<k \}$, where $k\geq 1$ is an integer.  A
slab is in level $k$ if it lies in $Y_{k}$; the zeroth level is empty.
A slab $S$ is a parent of a slab $S'$ if $S$ sits in a lower level and
sits under $S'$.

\subsection*{(iii) Ballot sequences.}  The set $\sX_{nm}$ consists of
the ballot sequences $B = (a_{1}, \dotsc , a_{2nm})$.  Let $s_{k} =
\sum_{i=1}^{k} a_{i}$ be the $k$th partial sum.  Let $i<j$.  We say
$i, j$ are a \emph{pair} if (i) $a_{i} = 1$, $a_{j}=-1$; (ii)
$s_{i}=s_{j}+1$; and (iii) $j$ is the minimal index greater than $i$
for which these conditions are true.  Then the elements of a ballot
sequence are its pairs.  The level of a pair $(i,j)$ is the value of
$s_{i}$; the zeroth level is empty.  A pair $p=(i,j)$ is the parent of
$q=(k,l)$ if the level of $p$ is one less than that of $q$ and $i<k$
and $l>j$.

\subsection*{(iv) Binary plane trees.}  The set $\sX_{nm}$ consists of
all binary plane trees with $nm$ vertices.  The zeroth level is empty.
The $i$th level for $i\geq 1$ consists of all vertices that can be
reached from the root by a path that always moves away from the root
with exactly $i-1$ left steps.  A vertex $v\in X_{i}$ is a parent
of $v'\in X_{i+1}$ if there is a path from $v$ to $v'$ with exactly
one left step.

\subsection*{(v) Triangulations.}  Let $\Pi = \Pi_{nm+2}$ be a regular
polygon with $nm+2$ sides.  Then $\sX_{mn}$ consists of triangulations
$\Delta$ of $\Pi$ that do not have new vertices.  The elements of
$\Delta$ are its triangles, and the levels of $\Delta$ are given by
the sets of \emph{left-turning triangles}, which means the following.
Fix once and for all an edge $e$ of $\Pi$ and let $T$ be the triangle
of $\Delta$ meeting $e$.  As one enters $T$ across $e$ there is a
unique exiting edge $e_{R}$ to the right and one $e_{L}$ to the left.
We say the triangle $T_{R}$ meeting $T$ at $e_{R}$, if there is one,
is obtained by turning right, and the triangle $T_{L}$ across $e_{L}$,
if there is one, is obtained by turning left.  Then the first level
$\Delta_{1}$ of $\Delta$ consists of $T$ and all the triangles that
can be reached from $T$ by turning left.  The second level
$\Delta_{2}$ consists of the triangles that can be reached by turning
right once from a triangle on the first level, and then turning left
an arbitrary number of times.  Continuing this process, each triangle
in $\Delta$ gets placed into a unique positive level.  The zeroth
level is empty.  A triangle $T\in \Delta_{i}$ is the parent of $T'\in
\Delta_{i+1}$ if $T'$ can be reached from $T$ by a single right turn
followed by any number of left turns.

\begin{figure}[htb]
\psfrag{e}{$e$}
\psfrag{tl}{$T_{L}$}
\psfrag{tr}{$T_{R}$}
\begin{center}
\includegraphics[scale=0.25]{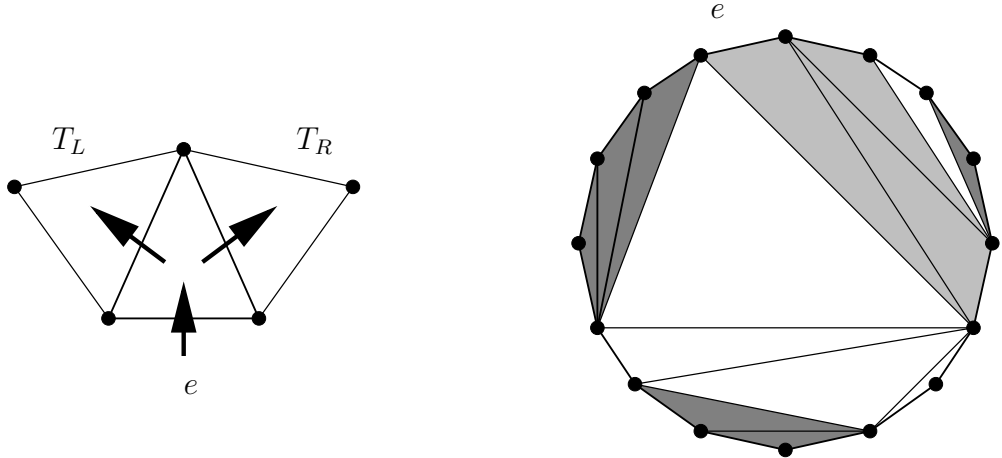}
\end{center}
\caption{The left and right triangles determined by an edge $e$, and a
triangulation of a polygon with its level structure.  The light grey
triangles, which are obtained by turning left after entering across
the edge $e$, are on level 1. The white triangles are on level 2, and
the dark grey triangles are on level 3.\label{fig:triangles}}
\end{figure}

\subsection*{(vi) Polygon gluings.}  Let $\Pi = \Pi_{2mn}$ be a
regular polygon with $2mn$ sides with a distinguished vertex.  An
oriented gluing of $\Pi$ is a partition of its sides into $n$ blocks
of size $2m$, with the sides oriented such that as one moves clockwise
around $\Pi $, the orientations in a given block \emph{alternate}
(cf.~the right column of Figure \ref{fig:c22b}).  We also assume that
the first edge in any of these blocks is oriented such that clockwise
is positive.  The labeling of the edges induces an equivalence
relation on the vertices of $\Pi$, after one performs the
identifications.  Then $\sX_{nm}$ is the set of such gluings with the
number of equivalence classes of vertices \emph{maximal}.  A vertex
$b$ is a child of $a$ if the edge joining them is positive from $a$ to
$b$.  The distinguished vertex is at level $0$, and the levels of the
others are determined by requiring that passing from parent to child
increases the level.

\subsection{} Figures \ref{fig:c22a}--\ref{fig:c22b} illustrate the
combinatorial interpretations used in the computation of $C^{(2)}_{2}
= 6$.  For this number each object that affords an admissible
$2$-labeling has at most two levels; we show the first level using
light grey and the second level using white.  Labelings are indicated
by the letters $a,b$.  Note that only one object has more than one
$2$-labeling, namely the one appearing in the first three lines of
each figure.

\begin{figure}[htb]
\psfrag{a}{$\scriptstyle a$}
\psfrag{b}{$\scriptstyle b$}
\psfrag{c}{$\scriptstyle c$}
\psfrag{d}{$\scriptstyle d$}
\psfrag{e}{$\scriptstyle e$}
\psfrag{f}{$\scriptstyle f$}
\psfrag{p1}{$\scriptstyle 1$}
\psfrag{m1}{$\scriptstyle -1$}
\psfrag{p2}{$\scriptstyle 1$}
\psfrag{m2}{$\scriptstyle -1$}
\includegraphics[scale=0.15]{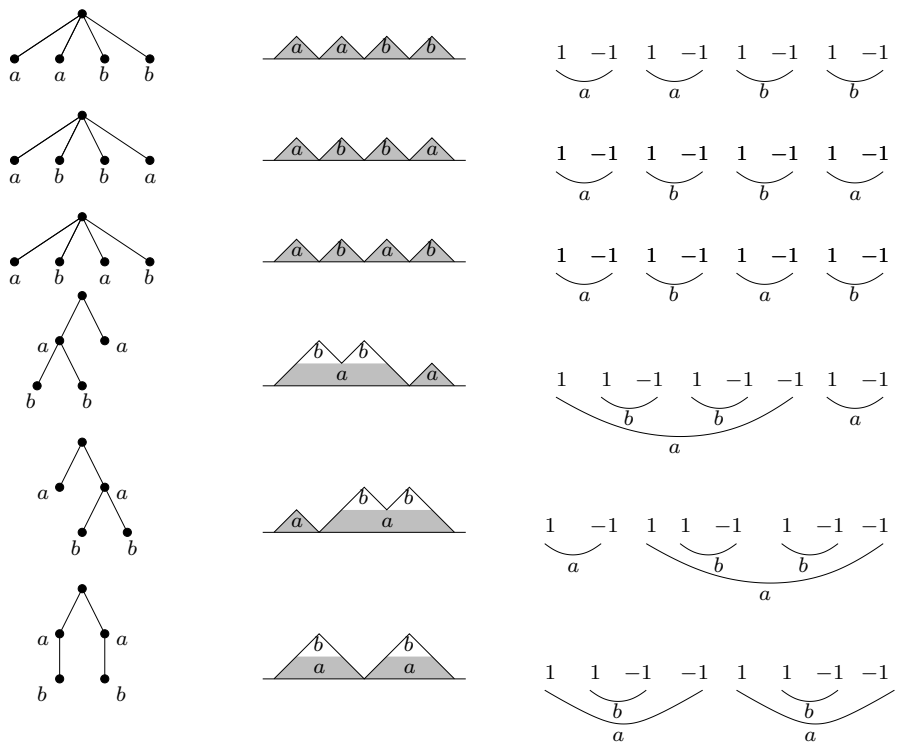}
\caption{Plane trees, Dyck paths, ballot sequences.\label{fig:c22a}}
\end{figure}

\begin{figure}[htb]
\psfrag{a}{$\scriptstyle a$}
\psfrag{b}{$\scriptstyle b$}
\psfrag{c}{$\scriptstyle c$}
\psfrag{d}{$\scriptstyle d$}
\psfrag{e}{$\scriptstyle e$}
\psfrag{f}{$\scriptstyle f$}
\psfrag{st}{$\scriptstyle *$}
\includegraphics[scale=0.15]{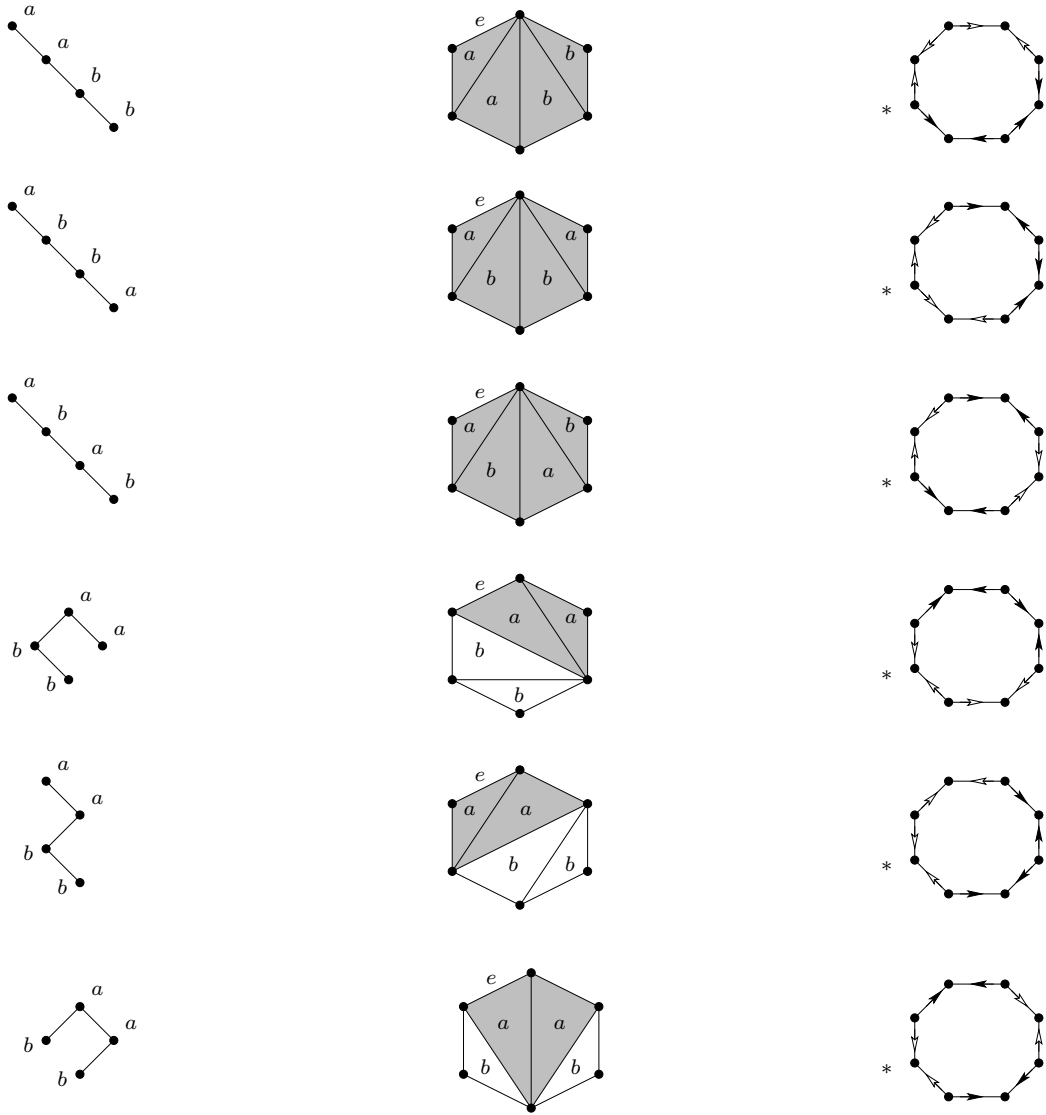}
\caption{Binary trees, triangulations, polygon gluings. \label{fig:c22b}}
\end{figure}

\begin{thm}\label{thm:interpretations}
The interpretations (i)--(vi) given in \S \ref{ss:interp} count the numbers
$C^{(m)}_{n}$, i.e.~\eqref{eq:wlabels} holds.
\end{thm}

\begin{proof}
First we claim that all the interpretations (i)--(vi) give the same
counts.  To see this, note that if $\sX_{nm}$ is any of the above sets
(with no labelings), then $|\sX_{nm}| = C_{nm}$, and there are known
bijections between the different objects (cf.~\cite[Theorem
1.5.1]{catalan}).  For the convenience of the reader, we recall these
bijections.  To simplify notation we set $m=1$.

\medskip\noindent \textbf{Plane trees and Dyck paths.}  Let $T$ be a plane
tree with $n+1$ vertices with root vertex $v$.  Then $T$ has $n$
edges, and one can build a length $2n$ Dyck path $\pi (T)$ as follows.
Recall that a plane tree either consists of the single vertex $v$, or
else $v$ has a sequence of subtrees $T_{1}, \dotsc , T_{k}$, each of
which is a plane tree.  To construct the Dyck path attached to $T$,
one performs the \emph{preorder tree traversal}: this is the traversal
that begins by visiting the root vertex, then recursively visits the
subtrees of the root $T_{1}, \dotsc , T_{k}$, in order.  The traversal
returns to the root along the root-incident edge every time it is done
traversing a subtree.

During the traversal one crosses each edge of $T$ exactly twice, once
going down, and once going up.  For every move down one appends the
step $(1,1)$ to $\pi (T)$, and for every move up one appends $(1,-1)$.
It is clear that the result is a Dyck path; it has $2n$ steps, never
goes below the $x$-axis, and ends at $(2n,0)$.  One can also easily
reverse this process: given a Dyck path $\pi$, one builds a plane tree
$T (\pi)$ whose preorder traversal is encoded by $\pi$.

\medskip \noindent \textbf{Dyck paths and ballot sequences.}  Let
$\pi$ be a Dyck path of length $2n$.  We create a sequence $b (\pi)=
(a_{1},\dotsc ,a_{2n})$ with $a_{i}\in \{\pm 1 \}$ by projecting onto
the second coordinate: $(1,1)\mapsto 1$ and $(1,-1)\mapsto -1$.  Thus
$b (\pi)$ records the change in height above the $x$-axis as one moves
along $\pi$.  Then $b (\pi)$ is clearly a ballot sequence: the
condition that $\pi$ never goes below the $x$-axis ensures that the
partial sums are nonnegative.  It is also clear that this process can
be reversed to give a map from ballot sequences to Dyck paths.

\medskip
\noindent \textbf{Plane trees and polygon gluings.}  Let $T$ be a
plane tree.  We can regard $T$ as embedded in the sphere $S^{2}$,
for example by applying stereographic projection to the plane.  If one
cuts the sphere open along the edges of $T$, one obtains a polygon
$\Pi_{2n}$ with $2n$ sides together with data of an oriented gluing: the
edges come naturally in pairs, and the root becomes the distinguished
vertex of $\Pi_{2n}$.  To go backwards, if one starts with a polygon
$\Pi_{2n}$ and identifies the edges in pairs, one obtains a
topological surface $S$ together with an embedded connected graph $G$ ($G$ may
have loops or multiple edges).  The surface $S$ is orientable if and
only if the gluing data is oriented (as one goes around the boundary
of the polygon, every pair of edges to be glued must appear with
opposite orientations).  If $S$ is orientable, then the graph $G$ is a
tree if and only if $S$ is a sphere; this happens if and only if the
number of vertices of $G$ is maximal after gluing;
this follows from considering the Euler characteristic of $S$,
computed as $1-|\text{edges}|+|\text{vertices}|$.\footnote{As
mentioned before (\S\ref{ss:introinterps}), there
is a connection between polygon gluings and interpretation (59) in
\cite[Chapter 2]{catalan}.  This interpretation is that $C_{n}$ is the
number of ways to draw $n$ nonintersecting chords joining $2n$ points
on the circumference of a circle.  If one starts with a connected
oriented polygon gluing, and draws chords between the centers of the edge
pairs, one obtains $n$ chords as in (59).  The condition that these
chords do not intersect is exactly equivalent to the resulting surface
being a sphere.}

\medskip \noindent \textbf{Triangulations and binary plane trees.}
Let $\Delta$ be a triangulation of the polygon $\Pi_{n+2}$ with a
distinguished edge $e$.  One can make a binary plane tree $B (\Delta)$
by taking the dual of $\Delta $ as follows (cf.~Figure
\ref{fig:tri2bin}).  The vertices of $B (\Delta)$ are the triangles of
$\Delta$.  Two vertices are joined by an edge if and only if they
correspond to adjacent triangles in $\Delta$.  The distinguished edge
$e$ sits in the boundary of one triangle, which determines the root of
$B (\Delta)$ (in Figure~\ref{fig:tri2bin} the root is indicated by a
doubled circle).  The resulting tree is embedded in the plane and is a
binary tree exactly because each triangle has three sides.  It is also
easy to see that this construction is reversible.

\medskip \noindent \textbf{Plane trees and binary plane trees.}
Finally we come to this bijection, due to de~Bruijn--Morselt
\cite{dbm}, which is the most interesting of all.  We follow the
presentation in the proof of \cite[Theorem 1.5.1]{catalan}
closely.\footnote{In particular we follow the paragraph labeled
(iii)$\rightarrow$(ii) at the bottom of p.~8.}

Let $T$ be a plane tree on $n+1$ vertices.  We will construct a binary
tree $B (T)$ on $n$ vertices; in fact the vertices of $B (T)$ are the
non-root vertices of $T$.  First we delete the root from
$T$ and all edges incident to it.  Then we remove all edges that are
not the leftmost edge from any vertex.  In other words, if the
children of $v$ in $T$ are $v_{1},\dotsc ,v_{k}$, then we remove the
edges $\{v, v_{2}\})$, \dots , $\{v, v_{k} \}$ and leave the edge
$\{v, v_{1} \}$ untouched.  After this the remaining edges become the
\emph{left} edges in the binary tree $B (T)$. 

To construct the \emph{right} edges in $B (T)$, we create edges
horizontally across the vertices of $T$.  In particular, given a
vertex $v$ in $T$, we draw edges from each child $w$ of $v$ to the
child of $v$ immediately to the right of $w$, if this child exists.
Finally the root of $B (T)$ is the leftmost child of the root of $T$.
(See Figure \ref{fig:plane2bin} for an example.)  This process is
easily seen to be reversible.

\begin{figure}[htb]
\psfrag{e}{$e$}
\begin{center}
\includegraphics[scale=0.25]{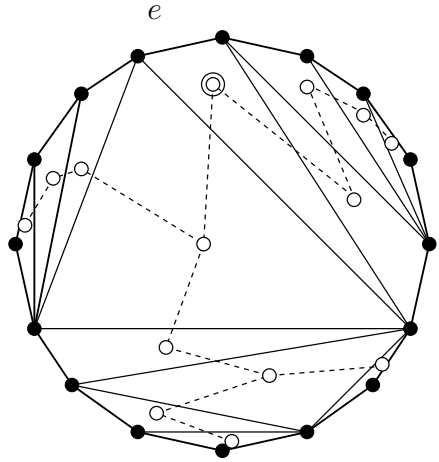}
\end{center}
\caption{Making a binary tree $B (\Delta)$ from a triangulation
$\Delta$.  The vertices of $B (\Delta)$ are white and its edges are
dashed.  The root vertex is circled.\label{fig:tri2bin}}
\end{figure}

\begin{figure}[htb]
\psfrag{arr}{$\longrightarrow $}
\begin{center}
\includegraphics[scale=0.25]{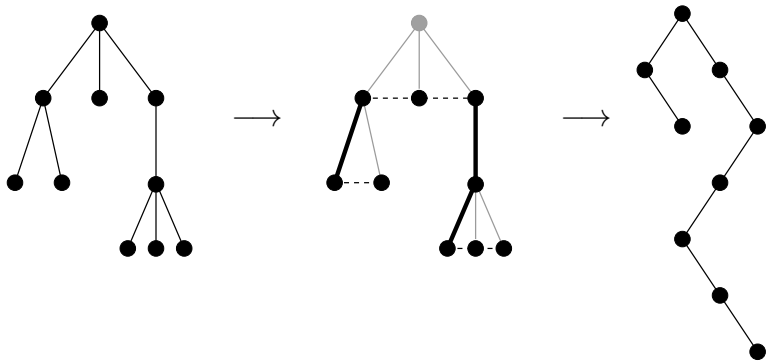}
\end{center}
\caption{Making a binary tree $B (T)$ from a plane tree $T$.  We first
delete the root of $T$ and all edges incident to it.  Then we delete
the edges in $T$ that are not furthest to the left.  The deleted edges
are shown in gray, and the remaining edges are drawn with a wider pen
width.  We then create horizontal edges, shown as dashed lines, by
drawing horizontally from the leftmost child of a vertex $v$ through
the remaining children of $v$, in order.  The original (respectively,
horizontal) edges become the left (resp., right) edges of $B
(T)$.\label{fig:plane2bin}}
\end{figure}

\medskip We now return to the original discussion. We claim that for
$m>1$, the same bijections prove that all interpretations give the
same counts.  One needs only to check that the definitions of
admissible labelings are compatible.  For instance, to pass from a
plane tree $T$ to a Dyck path $\pi $, one starts at the root of the
$T$ and traverses it in preorder (process the node, traverse the left
subtree, then traverse the right subtree).  As one descends $T$ along
an edge, one steps up in $\pi$ by $(1,1)$, and as one ascends along an
edge, one steps down by $(1,-1)$.  With this correspondence, it is
evident that the non-root vertices of $T$ coincide with the slabs of
$\pi$.  Indeed, the slabs of $\pi$ are constructed so that the
corresponding vertices of $T$ are at the same distance from the root,
and the poset structure in the slabs mirrors that of the vertices of
$T$.  It is also easy to see that the different notions of admissible
labelings agree.  The verifications for the other interpretations are
similar.

To complete the proof, we must check that any one of them actually
computes $C_{n}^{(m)}$.  We will use admissibly labeled plane trees
(i). 

Let $\sC = \sC_{n+1}$ be the set of pairs 
\[
\Bigl\{(T,w) \Bigm | \text{$T \in \TT_{n+1}$, $w$ an $a_{T}^{(m)}$-tour}\Bigr\}
\]
modulo the equivalence relation $(T,w) \sim (T',w')$ if $T=T'$ and
there is an automorphism of $T$ taking $w$ to $w'$.  Then $|\sC | =
C^{(m)}_{n}$.  Let $\sP = \sP_{n}$ be the set of plane trees on $n$
vertices, and let $\sA = \sA_{nm+1}$ be the set of pairs
\begin{equation}\label{eq:defofA}
\Bigl\{(T,l) \Bigm |  \text{
$T\in \sP_{nm+1}$ and $l$ is an admissible
$m$-labeling of $T$}\Bigr\}.
\end{equation}
Then $|\sA |$ is the right hand side of \eqref{eq:wlabels}.  We will
prove that \eqref{eq:wlabels} holds by constructing a bijection
between $\sC$ and $\sA$. 

We begin by defining two maps 
\[
\alpha \colon \sC \longrightarrow \sA , \quad \beta \colon \sA \longrightarrow \sC.
\]
First we define $\alpha$.  Given $(T,w) \in \sC$, let $v$ be the
vertex where $w$ begins and ends.  Then $\alpha (T,w)$ is the plane
tree $\tilde{T}$ whose Dyck word $\pi (\tilde{T})$ is determined by
the following rule: the $i$th step of $\pi (\tilde{T})$ is an up-step
if the $i$th step of $w$ moves towards $v$, and is a down-step
otherwise.  We equip $\tilde{T}$ with a labeling $l$ as follows: the
$i$th vertex in preorder traversal of $\tilde{T}$ is given the label
of the starting point of the $i$th edge of $w$.  This labeling $l$ is
an admissible $m$-labeling: if $x_{1},x_{2}$ are two vertices of
$\tilde{T}$ with the same labels, then the parents $y_{1},y_{2}$ of
$x_{1}, x_{2}$ receive the label of the unique vertex $y\in T$ that is
one step closer to $v$.  Thus $\alpha (T,w)\in \sA$.  We also claim
$\alpha$ is well-defined.  Suppose $\alpha (T,w) = (\tilde{T},l)$ and
$(T,w)\sim (T',w')$.  Then $\alpha (T',w')$ is the pair $(\tilde{T},
l')$, where the labels of $l'$ are a permutation of those of $l$.
Figure \ref{fig:alpha} shows an example of the computation of
$\alpha$ for $m=2$.

Now we define $\beta $.  Let $(S,l)\in \sA$.  Let $\bar S$ be the
graph obtained from $S$ by first identifying any two vertices of $S$
that share the same label, and then by replacing parallel edges in the
resulting multigraph by single edges.  We note that $\bar S$ is
actually a tree.  Indeed, the admissibility of the labeling implies
that the map $E (S) \rightarrow E (\bar S)$ on edges is $m : 1$, and
the map on vertices is $m:1$ away from the root and $1 : 1$ on the
root.  Thus $\bar S$ has $n+1$ vertices and $n$ edges, and since it is
clearly connected $\bar S$ is a tree.  Finally, we define a walk $w = w
(\bar S)$ on $\bar S$ by writing the sequence of labels encountered
during the preorder traversal of $S$.  The $m$-admissibility of the
labeling implies that $w$ is an $a^{(m)}_{S}$-tour, and thus
we have defined an element $\beta (S,l) = (\bar S, w)$.  Figure
\ref{fig:beta} shows an example of this construction.

To complete the proof, we must show that $\alpha$, $\beta$ are
bijections.  First we show $\beta \circ \alpha = 1_{\sC}$.  We claim
that if $(\bar S, w (\bar S)) = (\beta \circ \alpha) (T,w)$, then
$\bar S = T$ and $w (\bar S) = w$ (in other words, the representative
of an equivalence class in $\sC$ is taken to itself).  This is clearly
true if $T$ has one vertex.  Suppose the claim is true for all pairs
$(T,w)$ with $\leq n$ vertices.  Let $T$ be a tree with $n+1$
vertices, let $w$ be an $a_{T}^{(m)}$-tour of $T$ beginning at $v$,
let $x$ be a leaf of $T$, and let $y$ be the neighbor of $x$ in $T$.
Deleting $x$ and modifying $w$ appropriately we obtain a pair $(T',
w')$.  In $(S', l')= \alpha (T',w')$ we have $m$ vertices
$y_{1},\dotsc ,y_{m}$ corresponding to $y$.  Then $(S,l)= \alpha
(T,w)$ is obtained from $(S',l')$ by placing $m$ new vertices
$x_{1},\dotsc ,x_{m}$ under the $y_{i}$ according to where they appear
in the walk $w$.  By induction $\beta (S',l') = (T',w')$.  When we
construct $\beta (S,l)$ the only differences are that now we collapse
the $m$ edges of the form $(x_{i},y_{j})$ to a single edge in $\bar S$
joining $x$ to $y$, and that we build $w (\bar S)$ from $w (\bar S')$
by incorporating the vertex $x$.  Clearly $\bar S = T$.  Moreover $w
(\bar S) = w$, since the edges in $S$ were built to make preorder
traversal in $S$ match the original walk $w$.  This shows $\beta \circ
\alpha = 1_{\sC}$.

Now we show $\alpha \circ \beta = 1_{\sA}$.  The proof is similar.
Clearly this is true for all $(S,l)\in \sA$ with one vertex.  Suppose
that for all $(S,l)$ with $<nm+1$ vertices we have $(\alpha \circ
\beta) (S,l) = (S,l)$.  Let $(S,l)$ have $nm+1$ vertices and choose a
leaf $x_{1}$ of $S$ of highest level.  The admissibility of the
labeling $l$ implies that there are $m-1$ additional vertices
$x_{2},\dotsc ,x_{m}$ with the same label as $x_{1}$, and since
$x_{1}$ lies in the highest level, the vertices $x_{2},\dotsc ,x_{m}$
must also be leaves of $S$.  Let $y_{1},\dotsc ,y_{m}$ be the vertices
of $S$ with label equal to any parent of $x_{i}$.  In particular, the
parents of the $x_{i}$ are a subset of the $y_{j}$, but not every
$y_{j}$ is necessarily a parent of an $x_{i}$.  Deleting the $x_{i}$
we obtain a labeled plane tree $(S',l')$.  If $(\bar S', w (\bar S'))
= \alpha (S',l')$, then by induction we have $\beta (\bar S', w (\bar
S')) = (S',l')$, and moreover the tree $\bar S'$ is obtained from
$\bar S$ by deleting $x$ (the image of the $x_{i}$).  The walk $w
(\bar S)$ is obtained from $w (\bar S')$ by inserting steps along the
edge $(x,y)$ exactly according to the up- and down-steps in $S$ along
the edges of the form $(x_{i},y_{j})$, where $y$ is the image of the
$y_{j}$.  After applying $\alpha$ we exactly recover the edges
$(x_{i},y_{j})$ in $S$.  This shows $\alpha \circ \beta = 1_{\sA}$,
and completes the proof of the theorem.
\end{proof}

\begin{remark}
The interpretations in Theorem \ref{thm:interpretations} make it
possible to define various higher analogues of other standard numbers,
such as Narayana numbers, and higher $q$-analogues.  We have not
pursued these definitions. 
\end{remark}

\begin{figure}[htb]
\begin{center}
\psfrag{a}{$\scriptstyle a$}
\psfrag{b}{$\scriptstyle b$}
\psfrag{c}{$\scriptstyle c$}
\psfrag{d}{$\scriptstyle d$}
\psfrag{v}{$\scriptstyle v$}
\psfrag{T}{$T$}
\psfrag{T'}{$\tilde{T}$}
\psfrag{ar}{$\longrightarrow$}
\includegraphics[scale=0.25]{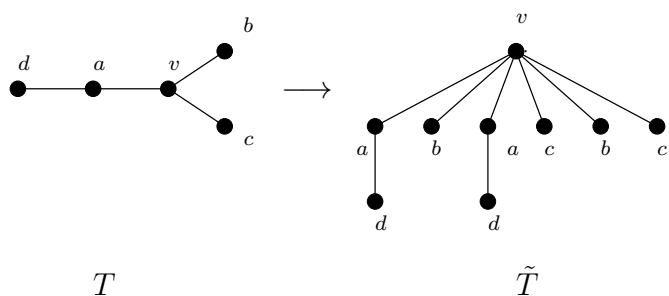}
\end{center}
\caption{Applying the map $\alpha$ to a pair $(T,w)$ gives an
admissibly $m$-labeled plane tree $\tilde{T}$.  The label set is
$\{a,b,c,d \}$, and the walk is $w = vadavbvadavcvbvcv$. If we apply the
automorphism of $T$ that swaps vertices $b$ and $c$, we obtain an
equivalent admissibly $m$-labeled plane tree. \label{fig:alpha}}
\end{figure}

\begin{figure}[htb]
\begin{center}
\psfrag{a}{$\scriptstyle a$}
\psfrag{b}{$\scriptstyle b$}
\psfrag{c}{$\scriptstyle c$}
\psfrag{d}{$\scriptstyle d$}
\psfrag{v}{$\scriptstyle v$}
\psfrag{T}{$S$}
\psfrag{T'}{$\bar{S}$}
\psfrag{ar}{$\longrightarrow$}
\includegraphics[scale=0.25]{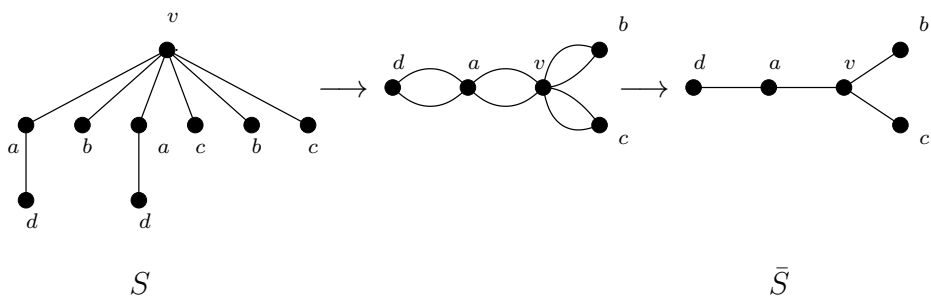}
\end{center}
\caption{The map $\beta$ is the composition of these two arrows;
applied to a pair $(S,l)$ gives a tree $\bar {S}$ and a walk $w$.  The
resulting walk is the same as that in Figure
\ref{fig:alpha}. \label{fig:beta}}
\end{figure}

\section{Generating functions and asymptotics}\label{s:gf}

\subsection{}
Let 
\begin{equation}\label{eq:Fm0}
F_{m} (x) = \sum_{n\geq 0} C^{(m)}_{n}x^{n}
\end{equation}
be the ordinary generating function of the $C^{(m)}_{n}$.  In this
section we explain how to use the combinatorial interpretations in
\S\ref{s:ci} to compute \eqref{eq:Fm0} to arbitrary precision as a
power series in $x$.  Then we will give a conjecture of the asymptotic
behavior of $C^{(m)}_{n}$ that generalizes the famous formula 
\[
C_{n} \sim \frac{4^{n}}{n^{3/2}\sqrt{\pi}}, \quad (n\rightarrow \infty).
\]

\subsection{}
We will compute $F_{m} (x)$ by showing that counting admissibly
labeled plane trees on $nm+1$ vertices is equivalent to counting
colored plane trees on $n+1$ vertices.  To count the latter, we use
standard generating function techniques as described in
Flajolet--Sedgewick \cite{fs}.  For the benefit of the reader and to
keep our presentation as self-contained as possible, we recall the
results we need here.

\begin{prop}\label{prop:coloringtrees}
\leavevmode
\begin{enumerate}
\item Let $A (x) = \sum_{k\geq 0} a_{k}x^{k} \in \ZZ \sets{x}$ be an
integral formal power series with $a_{k}\geq 0$.  Let $\sP_{A}$ be the
set of all colored plane trees such that any vertex with $k$ children
can be painted one of $a_{k}$ colors.  Let $P_{A}\in \ZZ \sets{x}$ be
the ordinary generating function of $\sP_{A}$, so that the coefficient
$[x^{n}]P_{A} (x)$ of $x^{n}$ counts the trees in $\sP_{A}$ with $n$
vertices.  Then $P_{A}$ satisfies the functional relation
\begin{equation}\label{eq:Aeqn}
P_{A} = x A (P_{A}).
\end{equation}
\item Let $B (x) = \sum_{k\geq 0} b_{k}x^{k} \in \ZZ \sets{x}$ be
another integral formal power series with $b_{k}\geq 0$.  Let
$\sP_{A,B}$ be the set of all colored plane trees such that (a) any
non-root vertex with $k$ children can be painted one of $a_{k}$
colors, and (b) if the root vertex has degree $k$ then it can be
painted any one of $b_{k}$ colors.  Let $P_{A,B}\in \ZZ \sets{x}$ be
the ordinary generating function of $\sP_{A,B}$.  Then $P_{A,B}$
satisfies the functional relation
\begin{equation}\label{eq:Beqn}
P_{A,B} = x B (P_{A}),
\end{equation}
where $P_{A}$ satisfies \eqref{eq:Aeqn}.
\end{enumerate}
\end{prop}

\begin{proof}
These results are proved in \cite{fs}.  Specifically (i) is
\cite[Proposition I.5]{fs} and (ii) follows from \cite[Example
III.8]{fs}.  We sketch the proof here.

For \eqref{eq:Aeqn}, recall that a plane tree is a rooted tree with an
ordering specified for the children of each vertex.  This is
equivalent to the recursive specification that a plane tree is a root
vertex $v$ with a (possibly empty) sequence of plane trees $T_{1},
\dotsc , T_{k}$ attached in order to $v$, with the root of $T_{i}$
becoming the $i$th child of $v$.  If $P$ is the ordinary generating
function of plane trees, this description leads to the well-known
functional relation
\[
P = x (1+P+P^{2}+P^{3}+\dotsb ),
\]
where the right hand side encodes the process of attaching an
arbitrary sequence of plane trees to an initial root vertex. 

Now consider the collection $\sP_{A}$ of plane trees colored according
to $A$.  In terms of the recursive specification above, one can think
of building a tree in $\sP_{A}$ by (i) picking a root vertex $v$ and
deciding on the number $k$ of its children, (ii) coloring $v$ one of
$a_{k}$ different colors, and (iii) continuing the process recursively
to each of the $k$ children of $v$.  This leads to the functional
relation 
\[
P_{A} = x (a_{0}+a_{1}P_{A}+a_{2}P_{A}^{2}+a_{3}P_{A}^{3}+\dotsb ),
\]
which proves \eqref{eq:Aeqn}.

To see \eqref{eq:Beqn}, we note that a tree in $\sP_{A,B}$ is built
similarly.  We (i) choose a root vertex $v$ and
decide on the number $k$ of its children, (ii) color $v$ one of
$b_{k}$ different colors, and (iii) place trees from $\sP_{A}$ under
$v$ with their roots as $v$'s children.  This leads to
\[
P_{A,B} = x (b_{0}+b_{1}P_{A}+b_{2}P_{A}^{2}+b_{3}P_{A}^{3}+\dotsb ),
\]
which proves \eqref{eq:Beqn}.
\end{proof}

\begin{defn}\label{def:abcoloring}
If $S$ is a plane tree appearing in $\sP_{A,B}$, then we say $S$ has
been equipped with an \emph{$(A,B)$-coloring}.
\end{defn}

\subsection{}
Now we define the generating functions for the colorings we will need.
For any pair $g\geq 0, r\geq 1$, let $\lambda (r,g)$ be the dimension of the space
of degree $g$ homogeneous polynomials in $r$ variables.  We have
\[
\lambda (r,g) = \binom{r-1+g}{r-1}.
\]
For any $k\geq 0$, let $W_{m} (k)$ be the number of partitions of a
set of size $k$ into blocks of size $m$.  Then $W_{m} (k) = 0$
unless $m$ divides $k$, and in this case we have
\[
W_{m} (k) = \frac{k!}{(m!)^{k/m} (k/m)!}.
\]

\begin{thm}\label{thm:gf}
Put 
\[
\ell_{m} (x) = \sum_{d\geq 0} W_{m} (dm) \lambda (m,dm)x^{d}
\]
and 
\[
h_{m} (x) = \sum_{d\geq 0} W_{m} (dm) x^{d}.
\]
Let $f_{m} (x) \in \ZZ [\![x]\!]$ satisfy the functional equation 
\begin{equation}\label{eq:fm}
f_{m} (x) = x \ell_{m} (f_{m} (x)).
\end{equation}
Then the generating function $F_{m} (x) = \sum_{n\geq 0}
C^{(m)}_{n}x^{n+1}$ satisfies 
\begin{equation}\label{eq:Fm}
F_{m} (x) = x h_{m} (f_{m} (x)).
\end{equation}
\end{thm}

\begin{proof}
Let us fix $m\geq 1$ and simplify notation by writing $F$, $f$,
$\ell$, $h$ instead of $F_{m}$, $f_{m}$, $\ell_{m}$, $h_{m}$.  Then by
Proposition \ref{prop:coloringtrees} $F$ is the ordinary generating
function of the $(\ell ,h)$-colored plane trees $\sP_{\ell ,h}$.
Recall \eqref{eq:defofA} that $\sA$ is the set of pairs $(T,l)$ where
$T\in \sP_{nm+1}$ and $l$ is an admissible $m$-labeling of $T$.  We
will construct a surjective map
\begin{equation}\label{eq:rho}
\rho \colon \sA \longrightarrow \sP  
\end{equation}
show that it induces a bijection between trees $(\ell , h)$-colored
plane trees with $n+1$ vertices and admissibly $m$-labeled plane trees
$\sA_{nm+1}$ on $nm+1$ vertices.  First we assume that the 
labeling set $L$ is the positive integers $\ZZ_{\geq 0}$, with its
standard total order.  Then we assume that $(S,l)\in \sA$ is labeled so
that the following properties hold:
\begin{enumerate}
\item \label{it:prop1} If $S$ has $nm+1$ vertices, then $l$ is surjective onto $\sets{n}\subset L$.
\item If $x,y\in S$ are two vertices and $y$ is in a higher level than
$x$ (i.e.~$y$ is further from the root than $x$), then $l (x)<l (y)$.
\item If $x,x'\in S$ are on the same level with $l (x)<l (x')$, and
$y$ (respectively, $y'$) is a child of $x$ (resp., $x'$), then $l (y)<l (y')$.
\item \label{it:prop4} If $L'\subset L$ is the subset of labels appearing in a fixed
level of $S$, then the elements of $L'$ are ordered compatibly with
when they first appear when read from left to right: the leftmost
vertex receives the smallest label from $L'$, the first new label seen
when reading to the right is next largest available label from $L'$,
and so on until the final label in $L'$ is seen.
\end{enumerate}
It is clear that, given any admissibly $m$-labeled tree, one can
permute the labels to satisfy these properties, and that the resulting
labeling is unique.

Recall that in the proof of Theorem \ref{thm:interpretations}, we
built a topological tree $\bar S$ on $n+1$ vertices from $S$ by
identifying vertices with the same labels and replacing parallel edges
with single edges.  After fixing a total order on $L$ and requiring
that the labels of $S$ satisfy the above, we obtain a canonical plane
tree structure on $\bar S$.  The root of $\bar S$ is the image of the
root of $S$, and each non root vertex of $\bar S$ receives a unique
label from $1$ to $n$.  The children of any vertex are ordered using
the order in $\sets{n}\subset \ZZ_{>0}$.  In the resulting plane tree,
which we also denote by $\bar S$, the labels in a given level increase
when read from left to right, and labels in higher levels are larger
than those in lower levels.  In particular, the labeling of $\bar S$
is uniquely determined; for this discussion, we will say that the
plane tree $\bar S$ has been canonically labeled.  Figure
\ref{fig:rho} shows an example of an admissibly $2$-labeled plane tree
$S$ satisfying the conditions \eqref{it:prop1}--\eqref{it:prop4}, and
the resulting canonically labeled plane tree $\bar S$.

We have thus constructed the map $\rho$ in \eqref{eq:rho}.  We claim
$\rho$ is surjective.  We let $\bar S$ be a canonically labeled plane tree,
and we build an admissibly $m$-labeled tree $(S,l)\in \rho^{-1} (\bar
S)$ as follows.  Suppose $\bar v \in \bar S$ is the root with children
$x_{1},\dotsc ,x_{d}$, read from left to right.  Under the root $v \in
S$ we place $dm$ vertices
\begin{equation}\label{eq:newverts}
x_{1}^{(1)},\dotsc ,x_{1}^{(m)},\dotsc ,x_{d}^{(1)},\dotsc ,x_{d}^{(m)},
\end{equation}
from left to right with $l (x_{i}^{(j)}) = l (x_{i})$, and put $\rho
(x_{i}^{(j)}) = \rho (x_{i})$.  Now suppose $\bar v\in \bar S$ is a
non root vertex with $d$ children.  By induction on distance to the
root, we may assume that $\bar v$ has already been lifted to $m$
vertices $v^{(1)},\dotsc ,v^{(m)}$ with $l (v^{(i)})= l (\bar v)$, and
these vertices have been placed in their level in $S$ in some order.
We lift the children $x_{1},\dotsc ,x_{d}$ of $\bar v$ to $dm$
vertices as before, again with $l (x_{i}^{(j)}) = l (x_{i})$ and $\rho
(x_{i}^{(j)}) = \rho (x_{i})$, and we make them all children of
$v_{1}^{(1)}$ in the order \eqref{eq:newverts}.  Continuing in this
way we obtain an admissibly $m$-labeled $(S,l)\in \rho^{-1} (\bar S)$
with $l$ satisfying \eqref{it:prop1}--\eqref{it:prop4}, which shows
$\rho$ is surjective.

To complete the proof, we claim that the inverse image $\rho^{-1}
(\bar S)$ is in bijection with the $(\ell ,h)$-colorings of $\bar S$.
To see this we revisit the proof of surjectivity and see what
additional choices one could make along the way to build $(S,l)$.
First consider the root vertex $\bar v \in \bar S$ and suppose it has
$d$ children $x_{1},\dotsc ,x_{d}$.  As before we lift these children
to $dm$ vertices $x^{(1)}_{1},\dotsc ,x_{d}^{(m)}$ under the root
$v\in S$, but we are not obligated to order them as in
\eqref{eq:newverts}. We can use any order compatible with the induced
ordering of the labels of the $x_{1},\dotsc ,x_{d}$ and with the
requirements \eqref{it:prop1}--\eqref{it:prop4}.  In particular this
implies that these orders are in bijection with set partitions of
$\sets{dm}$ into $d$ blocks $B_{1},\dotsc ,B_{d}$ of size $m$: if the
vertices $x_{i}^{(j)}$ are ordered under $v$ as $ y_{1},\dotsc ,y_{dm}
$, then
\[
l (y_{i}) = k \quad \text{if and only if} \quad i \in B_{k}.
\]
Since there are $W_{m} (dm)$ such set partitions, we see that the root
must be colored according to $h$.

Now consider a non root vertex $\bar v$ with $d$ children
$x_{1},\dotsc ,x_{d}$.  Again by induction we know that the vertices
$v^{(1)},\dotsc ,v^{(m)}$ mapping to $\bar v$ have been placed in
order in their level in $S$.  To place the $dm$ vertices $x_{1}^{(1)},
\dotsc , x_{d}^{(m)}$ in their level we must do two things:
\begin{enumerate}
\item [(A)] We must choose an order of the $x_{i}^{(j)}$ in their level, and
this order must be compatible with $l (x_{i}^{(j)}) = l (x_{i})$ and
the requirements \eqref{it:prop1}--\eqref{it:prop4}.
\item [(B)] After fixing the order of the $x_{i}^{(j)}$, we must decide how
to place them under their (potential) parents $v^{(1)}, \dotsc , v^{(m)}$.
\end{enumerate}
As before (A) corresponds to a set partition of $\sets{dm}$ into $d$
blocks of size $m$.  The data in (B) corresponds to an order
preserving map from $\sets{dm}$ to $\sets{m}$: if $x\in \{x_{i}^{(j)}
\}$ is a child of $v \in \{v^{(j)} \}$, then any $x'>x$ (according to
the order chosen in (A)) cannot be a child of any $v'<v$.  Such maps
are counted by $\lambda (m,dm)$.  Indeed, a degree $dm$ monomial
$z_{1}^{e_{1}}\dotsb z_{m}^{e_{m}}$ encodes that the first $e_{1}$ in
$\sets{dm}$ map to $1\in \sets{m}$, the next $e_{2}$ map to $2$, and
so on.  This means the total number of choices is $W_{m} (dm) \cdot
\lambda (m,dm)$, which implies that any non root vertex with $d$
children must be colored according to $\ell$.

We have thus shown that any $(S,l)$ in the inverse image $\rho^{-1}
(\bar S)$ corresponds uniquely to an $(\ell ,h)$-coloring of $\bar S$.  By
Proposition \ref{prop:coloringtrees}, this completes the proof.
\end{proof}

\begin{figure}[htb]
\psfrag{1}{$\scriptstyle 1$}
\psfrag{2}{$\scriptstyle 2$}
\psfrag{3}{$\scriptstyle 3$}
\psfrag{4}{$\scriptstyle 4$}
\psfrag{5}{$\scriptstyle 5$}
\psfrag{6}{$\scriptstyle 6$}
\psfrag{7}{$\scriptstyle 7$}
\psfrag{8}{$\scriptstyle 8$}
\psfrag{ar}{$\longrightarrow$}
\psfrag{S}{$S$}
\psfrag{bS}{$\bar S$}
\begin{center}
\includegraphics[scale=0.25]{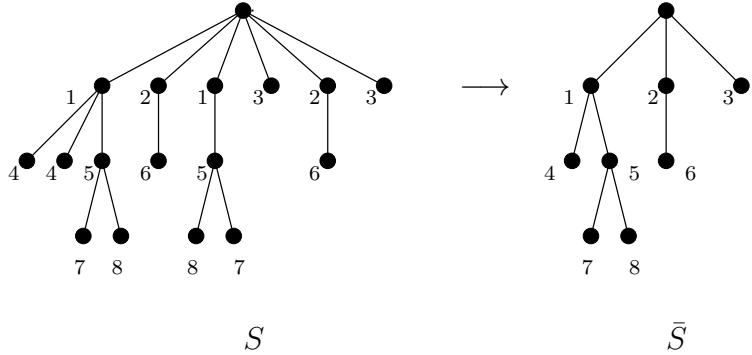}
\end{center}
\caption{The map $\rho$ takes the plane tree $S$ on $nm+1$ vertices to
a canonically labeled plane tree $\bar S$ on $n+1$
vertices.\label{fig:rho}}
\end{figure}

\begin{remark}
At this point the reader may be nostalgic for the recurrence
relation satisfied by the classical Catalan numbers, which corresponds
to the equation 
\begin{equation}\label{eq:rel}
x F_{1}^{2} - F_{1} + 1 = 0.
\end{equation}
For $m>1$, the power series $F_{m} (x)$ does not appear to be
algebraic, so unfortunately one does not have such a simple recurrence
relation on its coefficients.  However, experimentally one finds that
there is an algebraic relation satisfied by $f_{m}$ and $F_{m}$:
\begin{equation}\label{eq:rel2} 
f_{m}^{2} - xF_{m} +x = 0.
\end{equation}
Relation \eqref{eq:rel2} is easily checked when $m=1$.  We
have $f_{1} = x F_{1}$, so \eqref{eq:rel2} is really the same as
\eqref{eq:rel}.  For $m>1$ a proof of \eqref{eq:rel2} was given by
Mark Wilson \cite{mw}.  One observes that $\ell_{m}$ and $h_{m}$ are
related by 
\[
x \ell_{m} (x) =  h_m(x) - 1,
\]
from which \eqref{eq:rel2} easily follows.

The relation \eqref{eq:rel2} gives a connection between pairs of
objects computing hypergraph Catalan numbers in the spirit of that
encoded by \eqref{eq:rel}.  Consider $m=2$.  We have 
\[
f_{2} (x) = x + 3x^2 + 24x^3 + 267x^4 + \dotsb ,
\]
and 
\[
f_{2} (x)^{2} = x^2 +
6x^3 + 57x^4 + 678x^5 + \dotsb = x(F_{2} (x) - 1).
\]
We see $C_{3}^{(2)} = 57$ as the coefficient of $x^{4}$ in $f_{2}
(x)^{2}$, which comes from the coefficients of $x, x^{2}, x^{3}$ in
$f_{2} (x)$
via 
\begin{equation}\label{eq:rel3}
57 = 24\cdot 1 + 3\cdot 3 + 1\cdot 24.
\end{equation}
This is certainly reminiscent of the classical Catalan relation,
although there is an important difference.  The numbers involved on
the right of \eqref{eq:rel3} are connected with $C^{(2)}_{2},
C^{(2)}_{3}, C^{(4)}_{2}$, but they enumerate \emph{proper subsets} of
the associated objects, not the full sets.  Indeed, this is obviously
true, since both sides of \eqref{eq:rel3} involve the same number
$C^{(2)}_{3}$!  We have not attempted to explore this connection
further.
\end{remark}

\subsection{} We conclude this section by discussing asymptotics for
the $C^{(m)}_{n}$.  The results here are purely experimental; none
have been proved, although based on our numerical experiments we are
confident in them.  We learned this technique
from Don Zagier, who calls it \emph{multiplying by $n^{8}$}; an
excellent lecture by him at ICTP demonstrating the method can be
found online \cite{dz}.

Suppose one has a sequence $a = \{a_{n} \}_{n\geq 0}$ that
one believes satisfies an asymptotic of the form 
\begin{equation}\label{eq:an}
a_{n} \sim C_{0} + C_{1}/n + C_{2}/n^{2} + \dotsb \quad (n\rightarrow \infty).
\end{equation}
It may be difficult to extract $C_{0}$ to high precision even when $n$
is large, since $C_{1}/n$ might still be non-negligible.  However it
is possible to wash away the contributions of the nonconstant terms
$C_{k}/n^{k}$, $k\geq 1$.  One multiplies both sides of \eqref{eq:an}
by $n^{8}$ (or any other reasonable even power of $n$), and then
applies the difference operator $(8!)^{-1}\Delta^{8}$ to both sides,
where $\Delta a$ is the sequence $(\Delta{a})_{n} := a_{n+1}-a_{n}$.
The operator $\Delta^{k}$ annihilates any polynomial $p (n)$ of degree
$<k$, takes $n^{k}$ to $k!$, and takes $n^{-l}$ to a rational function
in $n$ of degree $-l-k$.\footnote{Here the 
degree of a rational function $P/Q$ in one variable means $\deg P - \deg Q$.}  Let $b
= \{b_{n} \}$ be the sequence resulting from multiplying $a_{n}$ by
$n^{8}/8!$ and applying the difference operator $\Delta$ eight times.
Then we have
\begin{equation}\label{eq:an2}
b_{n} \sim C_{0} + C_{9}p_{-9} (n) +
C_{10} p_{-10} (n) + \dotsb ,
\end{equation}
where $p_{-k} (n)$ denotes a rational function in $n$ of 
degree $-k$.  If one then evaluates the left of \eqref{eq:an2} at a
large value of $n$, the effects of $C_{k}$, $k\geq 9$ are negligible
on the right and one clearly sees $C_{0}$.  One can then repeat the
process with the sequence $n(a_{n}-C_{0})$ to find $C_{1}$, and so on.

\subsection{}
We illustrate with the series $F_{2} (x) = 1 + x + 6x^2 + 57x^3 +
678x^4 + \dotsb$.  Playing with the data one makes the ansatz 
\begin{equation}\label{eq:ansatz}
C^{(2)}_{n} \sim K A^{n}\,n! \,n^{\rho}
\end{equation}
for some constants $K, A, \rho$.  Indeed, apart from the $n!$, the
right of \eqref{eq:ansatz} is typical for these kinds of problems, and
was our initial guess; it quickly became evident that $n!$ needed to
be included.  The series 
\[
a_{n} := C^{(2)}_{n}/n!
\]
should then appear to grow exponentially, and the sequence of ratios $
\{a_{n}/a_{n-1}\}$ should satisfy \eqref{eq:an} with $C_{0} = A$.
Indeed, using 100 terms of $F_{2}$ and $\Delta^{16}$ we get
\[
A \approx
2.0000000000068961809\dotsb .
\]
Next we consider the sequence 
\[
a_{n} := C^{(2)}_{n}/(2^{n}\, n!),
\]
which we expect to be asymptotic to $K n ^\rho$.  We can detect $\rho$
using the sequence $b_{n} = (\Delta \log a)_{n}/ (\log (n+1) - \log
n)$, which satisfies \eqref{eq:an} with $C_{0} = -\rho$.  Again with
100 terms and $\Delta^{16}$ we get 
\[
-\rho \approx
0.499999997726715\dotsb .
\]
Hence we have
\[
C_{n}^{(2)} \sim K\cdot \frac{2^{n}\,n!}{\sqrt{n}}
\]
and we must find the constant $K$.  We consider $b_{n} =
C^{(2)}_{n}\sqrt{n}/2^{n}n!$ and look for $C_{0}$.  This time finding $K$ is
more difficult.  Taking 200 terms and applying $\Delta^{16}$,
this number appears to be $5.05704458036912766\dotsb $.  We use the
Inverse Symbol Calculator \cite{isc}, which attempts to symbolically
reconstruct a given real number using various techniques, and find
\[
K \approx 2 e^{3/2}/\sqrt{\pi}.
\]
(The $e^{3/2}$ is surprising, but is apparently correct.  Using 600
terms of the sequence, we find that $K\sqrt{\pi }/2$ agrees with
$e^{3/2}$ with relative error $<10^{-76}$.)  The conclusion is 
\[
C_{n}^{(2)} \stackbin{?}{\sim} e^{3/2} \cdot
\frac{2^{n+1}\,n!}{\sqrt{\pi n}}, \quad (n\rightarrow \infty).
\]

We present asymptotics for the $C^{(m)}_{n}$ as a conjecture:

\begin{conj}\label{conj:asymp}
Let $m>1$.  Then as $n\rightarrow \infty$, we have 
\[
C_{n}^{(m)} \sim K_{m} \cdot \frac{\bigl(\frac{m^{m-1}}{(m-1)!}\bigr)^{n+1}\,(n!)^{m-1}}{(\pi n)^{(m-1)/2}},
\]
where the constant $K_{m}$ is defined by 
\[
K_{m} = 
\begin{cases}
e^{3/2}&\text{if $m=2$,}\\
2\binom{2}{2}\binom{4}{2}\dotsb \binom{m-1}{2}/m^{(2m-3)/2}
&\text{if $m\geq 3$ and is odd, and}\\
\sqrt{2}\binom{3}{2}\binom{5}{2}\dotsb \binom{m-1}{2}/m^{(2m-3)/2}&\text{if $m\geq 4$ and is even.}
\end{cases}
\]
\end{conj}

\begin{remark}
We have tested Conjecture \ref{conj:asymp} numerically using 100 terms
of $F_{m} (x)$ for all $m\leq 30$.  We have not systematically tried to
find higher terms in the asymptotic expansion of $C_{n}^{(m)}$, as in
\eqref{eq:an}.
\end{remark}

\section{Connection with matrix models}\label{s:mm}

\subsection{} We finish by explaining how the numbers $C_{n}^{(m)}$
are related to hypergraphs and matrix models \cite{n_contribution}.
We first explain the connection between graphs, matrix models, and the
usual Catalan numbers.  For more information, we refer to
Harer--Zagier \cite{harer.zagier}, Etingof \cite[\S 4]{etingof},
Lando--Zvonkin \cite{lz}, and Eynard \cite{eynard}.

\subsection{}\label{ss:dmu2} Let $d\mu_{2} (x)$ be the measure on polynomial
functions on $\RR$ with moments
\[
\ip{x^{r}}_{2} := \int_{\RR} x^{r} d\mu_{2} (x) = W_{2} (r),
\]
where $W_{2} (r)$ is the number of pairings on a set of size $r$.
It is well-known that  $d\mu _{2} (x)$ is essentially the Gaussian
measure, up to normalization: we have 
\[
\ip{x^{r}}_{2} = (2\pi)^{-1/2} \int_{\RR} x^{r} e^{-x^{2}/2}\, dx,
\]
where $dx$ is the usual Lebesgue measure on $\RR$.

Let $g_{1}, g_{2}, \dotsc$ be a family of indeterminates, and let $S
(x)$ be the formal power series $\sum_{r\geq 1} g_{r}x^{r}/r!$.  We
can compute the expectation $\ip{\exp (S (tx))}_{2}$ as a formal power
series in $t$ with coefficients in the polynomial ring $\QQ [g_{1},
g_{2}, \dotsc ]$. We have
\begin{equation}\label{eq:ipSg}
\ip{\exp S (tx)}_{2} = 1 +A_{2}t^{2}/2 + A_{4}t^{4}/8
+A_{6}t^{6}/48 + \dotsb 
\end{equation}
where 
\begin{multline*}
A_{2} = g_{1}^{2}+g_{2}, \quad A_{4} = g_1^4 +
6g_{1}^{2}g_2 + 4g_{1}g_{3} + 3g_2^2 + g_4,\\
A_{6} = g_1^6 + 15g_{1}^{4}g_2 + 20g_1^3g_{3} + 
45g_{1}^{2}g_2^2 
+ 15g_{1}^{2}g_4 \\
+ 60g_{1}g_{2}g_{3} + 
6g_{1}g_5 + 15g_2^3 + 15g_{2}g_4 + 10g_3^2 + g_6.
\end{multline*}
The series \eqref{eq:ipSg} can be interpreted as a generating function
for graphs weighted by the inverse of the orders of their automorphism
groups (cf.~\cite[Theorem 3.3]{etingof}).  Let $\n = (n_{1},n_{2},\dotsc)$ be a vector of nonnegative
integers, with $n_{i}$ nonzero only for finitely many $i$.  Let $|\n |
= \sum n_{i}$. We say a graph $\gamma$ has \emph{profile $\n$} if it
has $n_{i}$ vertices of degree $i$.  Let $G (\n)$ be the set of all
graphs of profile $\n$, up to isomorphism (we allow loops and multiple
edges).  By an automorphism of a graph, we mean a self-map that
permutes edges and vertices.  In particular, automorphisms include
permuting multiedges between two vertices, and flipping loops at a
vertex (exchanging the two half-edges emanating from the vertex that
form the loop).  For any $\gamma \in G (\n)$,
let $\Gamma (\gamma)$ be its automorphism group.  Then we have
\[
\ip{\exp S (tx)}_{2} = \sum_{\n } t^{|\n |} \sum_{\gamma \in G (\n
)}\frac{\prod_{i} g_{i}^{n_{i}}}{|\Gamma (\gamma)|}.
\]
For example, consider the term $5g_{3}^{2}/24$ from $A_{6}$ in
\eqref{eq:ipSg}.  There are two graphs with this profile, shown in
Figure \ref{fig:2graphs}.  The left has $2\cdot 2\cdot 2$
automorphisms, and the right has $2\cdot 3!$, which gives $1/8+1/12 =
5/24$.

\begin{figure}[htb]
\begin{center}
\includegraphics[scale=0.3]{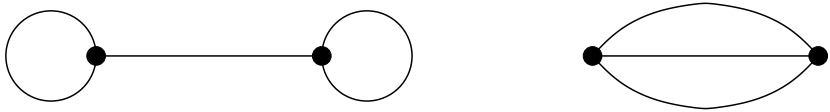}
\end{center}
\caption{The two graphs with profile $g_{3}^{2}$.\label{fig:2graphs}}
\end{figure}

\subsection{} Now we want to replace the Gaussian measure, which is
connected to counting pairings of a set, with something that is
connected to the numbers $W_{2m} (r)$, which count set partitions of
$\sets{r}$ with
blocks of size $2m$.  Let $d\mu_{2m} (x)$ be the
``measure'' on polynomial functions on $\RR$ that gives the monomial
$x^{r}$ the expectation $W_{2m} (r)$.  More precisely, we consider the
function taking $x^{r}$ to $W_{2m} (r)$ and extend linearly to
polynomials.  This is not a measure in the usual sense, although
formally we can regard it as such.  The ``expectation'' $\ip{\exp S
(tx)}_{2m}$ is then a well-defined power series in $t$, and has a
combinatorial interpretation via hyperbaggraphs.

Recall that a \emph{hypergraph} on a vertex set $V$ --- a notion due
to Berge \cite{berge} --- is a collection of subsets of $V$, called
the \emph{hyperedges}.  The degree of a vertex is the number of
hyperedges it belongs to, and a hypergraph is \emph{regular} if these
numbers are the same for all vertices.  The order of a hyperedge is
its number of vertices.  If all hyperedges have the same order, we say
that the hypergraph is \emph{uniform}.

Now suppose we allow $V$ to be a multiset, in other words a set with a
multiplicity map $V\rightarrow \ZZ_{\geq 1}$.  Then these
constructions lead to \emph{hyperbaggraphs}, due to Ouvrard--Le
Goff--Marchand{-}Maillet \cite{hbg}.(\footnote{In the CS literature,
multisets are sometimes called \emph{bags}.})  We extend the notions
of regularity and uniformity above by incorporating the multiplicity
in an obvious way (the order of a subset of a multiset is sum of the
multiplicities of its elements).

With these definitions, the expectation $\ip{\exp S (tx)}_{2m}$ now
enumerates uniform hyperbaggraphs of all profiles weighted by the
inverses of their automorphism groups, where each hyperedge has $2m$
elements.  For example,
\begin{equation}\label{eq:ipS}
\ip{\exp S (tx)}_{4} = 1 +B_{4}t^{4}/24
+B_{8}t^{8}/1152 + \dotsb 
\end{equation}
where 
\begin{equation*}
B_{4} = g_1^4 +
6g_{1}^{2}g_2 + 4g_{1}g_{3} + 3g_2^2 + g_4,\quad 
B_{8} = g_1^8 + 28g_1^6g_{2} + 56g_1^5g_{3} + \dotsb  + 35g_4^2 + g_8.
\end{equation*}
The computation of the contribution $35g_{4}^{2}/1152$ from $B_{8}$ in
\eqref{eq:ipS} is as follows.
There are three hyperbaggraphs of this profile, each with two
hyperedges.  The underlying set of vertices has $2$ elements $a$, $b$,
and we represent a hyperedge by a monomial in these variables.  The
profile $g_{4}^{2}$ means that each vertex has degree $4$, and since
$2m=4$ we must have uniformity $4$.  Thus we want pairs of monomials
in $a,b$ of total degree $4$.  This gives
\begin{equation}\label{eq:tup}
\{a^{4}, b^{4} \}, \quad \{a^{3}b, ab^{3} \}, \quad \{a^{2}b^{2}, a^{2}b^{2} \}.
\end{equation}
The orders of the automorphism groups are 
\begin{equation}\label{eq:aut}
2\cdot (4!)^{2}, \quad 2\cdot (3!)^{2}, \quad 2\cdot 2 \cdot (2!)^{2} (2!)^{2}.
\end{equation}
For example, the automorphisms of the last hyperbaggraph come from
interchanging the vertices, interchanging the two hyperedges, and the
internal flips within the hyperedges; the last type of automorphism
cannot occur for graphs.  Adding the inverses of these orders, one
finds $1/1152 + 1/72 + 1/64 = 35/1152$, which agrees with $B_{8}$
above.

As a final remark, we note that $A_{4} = B_{4}$.  This is a general
phenomenon: one can show that the coefficient of $t^{n}$ in $\ip{\exp
S (tx)}_{2m}$ is the \emph{complete
exponential Bell polynomial} $\mathbf{Y}_{n} (g_{1}, \dotsc , g_{n})$,
divided by $(2m)!^{d}\,d!$, where $d = n/2m$.
We refer to \cite[p.~134, eqn.~3b]{comtet} for the definition of
these; the coefficients of the $\mathbf{Y}_{n}$ can be found on OEIS
as sequence \texttt{A178867}.

\subsection{}
Now we pass to matrix models.  Let $V = V_{N}$ be the real vector space of $N\times
N$ complex Hermitian matrices.  The space $V$ has real dimension
$N^{2}$.  For
any polynomial function $f\colon V \rightarrow \RR $, define
\begin{equation}\label{eq:expectation}
\ip{f} = C^{-1}\int_{V} f (X) \exp (-\Tr X^{2}/2)\,dX,
\end{equation}
where $\Tr (X) = \sum_{i} X_{ii}$ is the sum of diagonal entries and
the constant $C$ is determined by the normalization $\ip{1} = 1$.  The
measure $\exp (-\Tr X^{2}/2)\,dX$ is essentially the product of the
Gaussian measures $d\mu_{2} (x)$ from \S\ref{ss:dmu2} taken over the real
coordinates of $V$.  The only difference is that for any off-diagonal
entry $Z_{ij} = X_{ij}+\sqrt{-1}Y_{ij}$, we have rescaled the measure
so that for even $r$ we have $\ip{X_{ij}^{r}}_{2} =
\ip{Y_{ij}^{r}}_{2} = W_{2} (r)/2^{r/2}$.

Now consider \eqref{eq:expectation}
evaluated on the polynomial given by taking the trace of the $r$th
power:
\begin{equation}\label{eq:basicint}
P(N,r) = \ip{\Tr X^{r}}.
\end{equation}
For $r$ odd \eqref{eq:basicint} vanishes for all $N$.  On the
other hand, for $r$ even and $N$ fixed, it turns out that $P
(N,r)$ is an integer, and as a function of $N$ is a polynomial of
degree $r/2+1$ with integral coefficients.

Furthermore, the number $P(N,r)$ has the following remarkable
combinatorial interpretation.  Let $\Pi_{r}$ be a polygon with $r$
sides.  Any pairing $\pi$ of the sides of $\Pi_{r}$ determines a
topological surface $\Sigma (\pi)$ endowed with an embedded graph (the
images of the edges and vertices of $\Pi_{r}$).  Let $v (\pi)$ be the
number of vertices in this embedded graph.  Then we have
\begin{equation}\label{eq:complexsurfsum}
P (N,r) = \sum_{\pi} N^{v (\pi)},
\end{equation}
where the sum is taken over all oriented pairings of the edges of
$\Pi_{r}$ such that the resulting topological surface
$\Sigma_{\pi}$ is orientable.  For example, we have
\begin{multline}\label{eq:polyexamples}
P(N,0) = N, P(N,2) = N^{2}, P(N,4) = 2N^{3}+N, \\
P(N,6) = 5 N^{4}+10N^{2}, \quad P( N,8) = 14N^{5}+70N^{3}
+ 21N.
\end{multline}
The pairings yielding $P(N,4)$ are shown in Figure \ref{fig:ccn4}.

\begin{figure}[htb]
\psfrag{n3}{$N^{3}$}
\psfrag{n}{$N$}
\begin{center}
\includegraphics[scale=0.4]{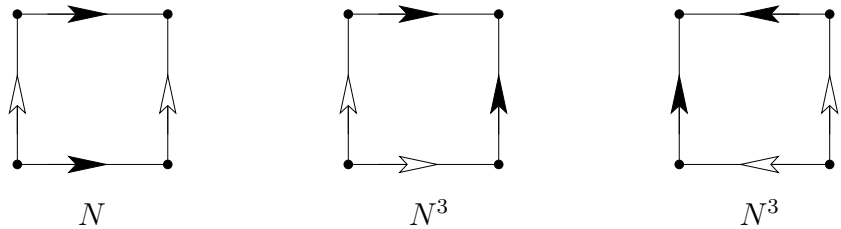}
\end{center}
\caption{Computing $P(N,4) = 2N^{3}+N$.}\label{fig:ccn4}
\end{figure}

\subsection{} One can see from \eqref{eq:polyexamples} that the
leading coefficient of $P (N,r)$ is none other than the Catalan number
$C_{r/2}$.  This follows easily from the interpretation of the Catalan
numbers in terms of polygon gluings ((vi) from \S\ref{ss:interp}).
Indeed, the leading coefficient counts the number of oriented pairings
of $\Pi_{r}$ such that the number of vertices $v (\pi)$ in the
orientable surface $\Sigma_{r}$ is \emph{maximal}; this is exactly the
interpretation above.

\subsection{} Now we modify the matrix model.  We replace the Gaussian
measure 
\[
\exp (-\Tr X^{2}/2)\,dX
\]
by the product of the formal measures $d\mu_{2m} (x)$ taken over the real
coordinates; again we rescale on the off-diagonal coordinates so that
for $r \equiv 0 \bmod 2m$ we have $\ip{X_{ij}^{r}}_{2m} =
\ip{Y_{ij}^{r}}_{2m} = W_{2m} (r)/2^{r/(2m)}$.  We write the
corresponding formal measure by $d\mu_{2m} (X)$.  Then we obtain a new
matrix model where polygons are glued by grouping their edges into
subsets of size $2m$ instead of pairs.  As above one can see that for
$r \equiv 0 \bmod 2m$, the integrals
\begin{equation}\label{eq:mtrace}
\int_{V} \Tr X^{r}\, d\mu_{2m}(X) 
\end{equation}
are polynomials $P_{2m} (N,r)$ in the dimension $N$.  For example, when $2m=4$ we have
\begin{multline*}
P_{4} (N,4) = N^{2}, \quad P_{4} (N,8) = 6N^{3}+21N^{2}+8N, \\
P_{4} (N,12) = 57N^{4}+715N^{3}+2991N^{2}+2012N.
\end{multline*}
The hypergraph Catalan numbers $C_{r}^{(m)}$ are the leading coefficients of these
polynomials:  we have
\[
P_{2m} (N,2mr) = C_{r}^{(m)} N^{r + 1} + \dotsb .
\]
A direct computation with the definition \eqref{eq:mtrace} shows that
these numbers are computed via Definition \ref{def:gencat}.  For more
details about these matrix models and the geometry of the polynomials $P_{2m}
(N,2mr)$, see \cite{n_contribution}.

\bibliographystyle{amsplain_initials_eprint_doi_url}
\bibliography{gencat}

\providecommand{\bysame}{\leavevmode\hbox to3em{\hrulefill}\thinspace}
\providecommand{\MR}{\relax\ifhmode\unskip\space\fi MR }
\providecommand{\MRhref}[2]{%
  \href{http://www.ams.org/mathscinet-getitem?mr=#1}{#2}
}
\providecommand{\href}[2]{#2}
\begin{thebibliography}{10}

\bibitem{berge}
C.~Berge, \emph{Graphs and hypergraphs}, North-Holland Publishing Co.,
  Amsterdam-London; American Elsevier Publishing Co., Inc., New York, 1973,
  Translated from the French by Edward Minieka, North-Holland Mathematical
  Library, Vol. 6.

\bibitem{isc}
J.~M. Borwein, P.~B. Borwein, and S.~Plouffe, \emph{Inverse symbolic
  calculator}, available at \url{http://wayback.cecm.sfu.ca/projects/ISC/}.

\bibitem{comtet}
L.~Comtet, \emph{Advanced combinatorics}, enlarged ed., D. Reidel Publishing
  Co., Dordrecht, 1974, The art of finite and infinite expansions.

\bibitem{dbm}
N.~G. de~Bruijn and B.~J.~M. Morselt, \emph{A note on plane trees}, J.
  Combinatorial Theory \textbf{2} (1967), 27--34.

\bibitem{n_contribution}
M.~DeFranco and P.~E. Gunnells, \emph{Hypergraph matrix models}, to appear in
  Moscow Math.~J.

\bibitem{etingof}
P.~Etingof, \emph{Mathematical ideas and notions of quantum field theory},
  available at \url{http://www-math.mit.edu/~etingof/18.238.html}, 2002.

\bibitem{eynard}
B.~Eynard, \emph{Counting surfaces}, Progress in Mathematical Physics, vol.~70,
  Birkh\"{a}user/Springer, [Cham], 2016, CRM Aisenstadt chair lectures.

\bibitem{fs}
P.~Flajolet and R.~Sedgewick, \emph{Analytic combinatorics}, Cambridge
  University Press, Cambridge, 2009.

\bibitem{harer.zagier}
J.~Harer and D.~Zagier, \emph{The {E}uler characteristic of the moduli space of
  curves}, Invent. Math. \textbf{85} (1986), no.~3, 457--485.

\bibitem{lz}
S.~K. Lando and A.~K. Zvonkin, \emph{Graphs on surfaces and their
  applications}, Encyclopaedia of Mathematical Sciences, vol. 141,
  Springer-Verlag, Berlin, 2004, With an appendix by Don B. Zagier,
  Low-Dimensional Topology, II.

\bibitem{massey}
W.~Massey, \emph{{Algebraic Topology: An Introduction}}, Graduate Texts in
  Mathematics, vol.~56, Springer--Verlag, 1977.

\bibitem{nauty}
B.~McKay, \emph{Nauty}, available at \url{http://pallini.di.uniroma1.it}.

\bibitem{hbg}
X.~Ouvrard, J.~L. Goff, and S.~Marchand{-}Maillet, \emph{Adjacency and tensor
  representation in general hypergraphs. part 2: Multisets, hb-graphs and
  related e-adjacency tensors}, CoRR \textbf{abs/1805.11952} (2018),
  arXiv:1805.11952.

\bibitem{oeis}
N.~J.~A. Sloane, \emph{{The On-Line Encyclopedia of Integer Sequences}},
  \url{oeis.org}.

\bibitem{addendum}
R.~P. Stanley, \url{http://www-math.mit.edu/~rstan/ec/catadd.pdf}.

\bibitem{ec2}
R.~P. Stanley, \emph{Enumerative combinatorics. {V}ol. 2}, Cambridge Studies in
  Advanced Mathematics, vol.~62, Cambridge University Press, Cambridge, 1999,
  With a foreword by Gian-Carlo Rota and appendix 1 by Sergey Fomin.

\bibitem{catalan}
R.~P. Stanley, \emph{{Catalan Numbers}}, Cambridge University Press, New York,
  2015.

\bibitem{stanley.alg.comb}
R.~P. Stanley, \emph{Algebraic combinatorics}, Undergraduate Texts in
  Mathematics, Springer, Cham, 2018, Walks, trees, tableaux, and more.

\bibitem{mw}
M.~Wilson, 2019, personal communication.

\bibitem{dz}
D.~Zagier, \emph{{ICTP Basic Notions Seminar Series: ``Asymptotics"}}, 2014,
  \url{http://www.youtube.com/watch?v=q8MGPjoyC9U}.

\end{thebibliography}

\end{document}